\documentclass[11pt,a4paper,twoside,draft]{article}
\usepackage{amssymb,amsthm}
\usepackage[namelimits,sumlimits,fleqn]{amsmath}
\usepackage{setspace}
\usepackage[a4paper,top=33mm, bottom=27mm, left=20mm, right=20mm]{geometry}
\usepackage[small, margin=20pt]{caption}
\usepackage{float}\restylefloat{figure}
\usepackage{url}
\usepackage[final]{graphicx}
\usepackage{subfig,multicol}

\usepackage{fancyhdr}
\allowdisplaybreaks[4]

\title{Matchings in Random Biregular Bipartite Graphs}
\author{Guillem Perarnau and Giorgis Petridis\\ \\
  \emph{Departament de Matem\`atica Aplicada $IV$.}\\ 
  \emph{Universitat Polit\`ecnica de Catalunya, BarcelonaTech.}\\
  \url{guillem.perarnau@ma4.upc.edu}\\ \\
  \emph{}\\
  \url{giorgis@cantab.net}
}

\date{}

\theoremstyle{plain}
\newtheorem{theorem}{Theorem}%[section]
\newtheorem{lemma}[theorem]{Lemma}
\newtheorem{proposition}[theorem]{Proposition}
\newtheorem{corollary}[theorem]{Corollary}

\theoremstyle{definition}
\newtheorem*{definition}{Definition}
\newtheorem*{acknowledgement}{Acknowledgement}
\newtheorem*{remark}{Remark}

\theoremstyle{remark}
\newtheorem*{claim}{Claim}

\newcommand{\rst}[1]{\ensuremath{{\mathbin\upharpoonright}%
\raise-.5ex\hbox{$#1$}}} % creates restriction symbol

\newcommand{\comment}[1]{\textbf{[#1]}} % comments 
\newcommand{\giorgis}[1]{{\bf [~Giorgis:\ } {\em #1}{\bf~]}} % comments by Giorgis

\newcommand{\whp}{\textit{whp\;}}

\newcommand{\I}{\mathcal{I}}

\newcommand{\Z}{\mathbb{Z}}
\newcommand{\Q}{\mathbb{Q}}
\newcommand{\G}{\mathcal{G}}
\newcommand{\E}{\mathbb{E}}
\newcommand{\eps}{\varepsilon}

\newcommand{\ga}{\Gamma} % creates neighbourhood symbol
\newcommand{\gai}{\Gamma^{-1}} % creates inverse neighbourhood symbol
\newcommand{\var}{\mathrm{Var}} % variance
\newcommand{\tends}{\rightarrow} % tendsto

\begin{document}

\pagenumbering{arabic}

\setcounter{section}{0}

\bibliographystyle{plain}

\maketitle

\onehalfspace

\thispagestyle{empty}

\begin{abstract}
We study the existence of perfect matchings in suitably chosen induced subgraphs of random biregular bipartite graphs. We prove a result similar to a classical theorem of Erd\H{o}s and R\'enyi about perfect matchings in random bipartite graphs. We also present an application to commutative graphs, a class of graphs that are featured in additive number theory.
\end{abstract}


\section{Introduction}\label{Intro}

Let us begin by defining the terms that appear in the title. Recall that, given two sets $A$ and $B$
of equal size and a bipartite directed graph on vertex set $(A,B)$,  a \emph{perfect matching} (also
known as a \emph{ $1$-factor}) from $A$ to $B$ is a collection of $|A|$ vertex disjoint edges from
$A$ to $B.$

\begin{definition}\label{G(k,n,d)}
Let $k\in \Q^+$ be a positive rational number, $n\in \Z^+$ a positive integer that satisfies $kn\in
\Z^+$ and $d\in \Z^+$ a positive integer that satisfies $1\leq d \leq n$ and $kd \in \Z^+.$ Let $Y$
be a set of size $n$ and $Z$ be a set of size $kn.$ Define $\G(k,n,d)$ to be the family of biregular
bipartite directed labelled graphs on the vertex set $(Y,Z)$ (with edges directed from $Y$ to $Z$)
where $d^+(y)=kd$ for all $y\in Y$ and $d^-(z)=d$ for all $z\in Z.$ A \emph{random biregular
bipartite directed graph} (with parameters $k,n,d$) is a graph chosen from $\G(k,n,d)$ uniformly at
random. The corresponding model of random graphs is denoted by $G(k,n,d).$
\end{definition}

The family $\G(k,n,d)$ is non-empty. We illustrate this by giving an example for integer $k$, which is
indicative of how biregular bipartite graphs are featured in additive number theory. We identify $Z$
with $\Z_{kn}$ and $Y$ with the subgroup $\{0,k,2k,\dots,(n-1)k\}.$ For $y\in Y$ and $z\in Z$ we
place an edge $yz\in E(G)$ if $z-y \in \{0,1,\dots kd-1\} \mod (kn).$ The resulting graph is a
member of $\G(k,n,d).$

The case where $k=1$ has a special relevance since $\G(1,n,d)$ is the family of
regular bipartite graphs of size $n$ and degree $d$ where the edges are canonically
oriented from one stable set to the other.
Estimating the size of $\G(1,n,d)$ as a function of $d$ and $n$ is a
question that has been studied extensively~\cite{Everett-Stein1971,Mineev-Pavlov1976}. 
Generalizations of this problem to biregular bipartite
graphs~\cite{McKay-Wang2003,Canfield-McKay2005} as well as to graphs with a prescribed sequence
of degrees in each of the stables have also been studied~\cite{McKay1984,McKay1985}.

% Mckay in~\cite{McKay1984} showed that there are \begin{equation}\label{number regular bipartite} \frac{(nd)!}{(d!)^{2n}}\exp\left(-\frac{(d-1)^2}{2}+O\left(\frac{d^3}{n}\right)\right)  \end{equation}  regular bipartite graphs of degree $d$, for any $d<n/6$. This result is asymptotically tight for $d=o(n^{1/3})$. In~\cite{McKay-Wang2003} the range of $d$ in~\eqref{number regular bipartite} is extended to $d=o(n^{1/2})$. Some results for larger $d$ can be found in~\cite{Canfield-McKay2005}.

%As we will see in the end of Sect.~\ref{Thm 2} 
Using Hall's theorem it is straightforward to check that every member of $\G(1,n,d)$ has a perfect
matching (see e.g.~\cite[Corollary $2.1.3$]{Diestel2005}).
For members of $\G(k,n,d)$ with $k\neq1$ there can be no perfect matching as the size of the two
layers is not equal. The distribution of the number of perfect matchings in random regular bipartite
graphs was studied by Bollob\'as and McKay in~\cite{Bollobas-McKay1986}, where its expected value
and variance are determined. 

% \giorgis{If you can find a reference for the fact that perfect matchings always exist, then we can
% shorten the paragraph and leave out the paragraph in Section 5}\guillem{It is in the Diestel's
% book, Corollary 2.1.3}

We tackle a different kind of question by studying the existence of a perfect matching in induced
subgraphs $H$ of members of $\G(k,n,d)$, whose stable sets have equal size. In particular
we determine how the probability of having such a perfect matching changes with $d$. Our result is
analogous to a classical result of Erd\H{o}s and R\'enyi.

Before stating the main result of the paper we recall that in any model of random graphs a property holds \emph{with high probability} if the probability that a random graph in the model satisfies this property tends to $1$ as $n$ tends to infinity.  From now on the phrase will be abbreviated to \textit{whp}, as it is common in the literature.

\begin{theorem}\label{G[A,B]}
Let $k\in \Q^+,$ $n\in \Z^+ $ be arbitrarily large and $d\in \{1,\dots,n\}$ and suppose that $kn, kd\in \Z^+$ with $kd\leq n.$  

Furthermore let $Y$ and $Z$ be sets of size respectively $n$ and $kn$ and $G\sim G(k,n,d)$. Take subsets
$A\subseteq Y$ and $B\subseteq Z$ of size $kd$ and define $H:=G[A,B]$ to be the subgraph induced by $G$ on vertex set $(A,B)$. Then
\begin{itemize}
\item[(i)] No perfect matching exists in $H$ \whp when $\frac{kd^2}{n}-\log(kd) \tends
-\infty$ or when $d$ is a constant.
\item[(ii)] A perfect matching exists in $H$ \whp when $\frac{kd^2}{n}-\log(kd) \tends
+\infty.$   
\end{itemize}
\end{theorem}

\begin{remark}
The second condition in conclusion $(i)$ has to be included because when $d$ is constant the quantity $\frac{kd^2}{n}-\log(kd)$ does not tend to $-\infty$.   
\end{remark}

Here and elsewhere, for any $y\in Y$ we define $\ga(y)=\{z\in Z : yz\in E(G)\}$ and for any $S\subseteq Y$, $\ga(S)=\cup_{y\in S} \ga(y).$ Similarly we define the inverse neighbourhood of $z\in Z$ by $\gai(z)$ and the inverse neighbouhood of $T\subseteq Z$ by $\gai(T)$.

The next result is a variation of Theorem~\ref{G[A,B]} when $B=\ga(y)$ for some $y\in A$.

\begin{theorem}\label{G[A,ga(y)]}
Let $k\in \Q^+,$ $n\in \Z^+ $ be arbitrarily large and $d\in \{2,\dots,n\}$ and suppose that $kn, kd\in \Z^+$ with $kd\leq n$. 

Furthermore let $Y$ and $Z$ be sets of size respectively $n$ and $kn$ and $G\sim G(k,n,d)$. Take a subset
$A\subseteq Y$ of size $kd$ and $y\in A$. Define $H:=G[A,\ga(y)]$ to be the subgraph induced by $G$ on vertex set $(A,\ga(y)).$  Then
\begin{itemize}
\item[(i)] No perfect matching exists in $H$ \whp when $\frac{kd^2}{n}-\log(kd) \tends
-\infty$ or when $d$ is a constant.
\item[(ii)] A perfect matching exists in $H$ \whp when $\frac{kd^2}{n}-\log(kd) \tends
+\infty.$   
\end{itemize}
\end{theorem}

The case $d=1$ is not covered by Theorem~\ref{G[A,ga(y)]}. It
is nonetheless easy to check that for $d=1$ a matching exists if and only if $k=1$. 

To put our results in context we briefly describe what holds in the most standard model of random directed bipartite graphs.

\begin{definition}
Let $A$ and $B$ be two sets of
size $n$. A \emph{random bipartite graph} with parameters
$n,p$ is a bipartite graph on the vertex set $(A,B)$ where edges are chosen
independently of each other with probability $p.$ The model of random bipartite graphs is
denoted by $B(n,p)$. 
\end{definition}

The existence of perfect matchings in random bipartite graphs was investigated by Erd\H{o}s and
R\'enyi about fifty years ago. They established the following in \cite{Erdos-Renyi1964}.

\begin{theorem}[Erd\H{o}s--R\'enyi]\label{Erdos-Renyi}
Let $c$ be a constant and $n$ an arbitrarily large positive integer. Furthermore let 
$$ 
p = \frac{\log{n} + c}{n}\;
$$
and consider a random bipartite graph $G'\sim B(n,p)$.

Then the probability there is a perfect matching in $G'$ is asymptotically equal to
$$
\Pr(\mbox{There exists a perfect matching in $G'$}) = (1+o(1)) \exp(-2 e^{-c})\;.
$$
In particular if $ np-\log{n} \tends +\infty$ when $n\tends + \infty,$ then there exists a
matching in $G'$ \textit{whp}; and if $np-\log{n} \tends -\infty$ when $n\tends
+\infty,$ then no matching exists in $G'$ \textit{whp}.
\end{theorem}
%\guillem{I've changed $d$ for $n$ in the Theorem, otherwise it looks quite strange.}

Theorem~\ref{G[A,B]} is an Erd\H{o}s--R\'enyi type result for the induced subgraph $H.$  To make the similarity between Theorem~\ref{G[A,B]} and Theorem~\ref{Erdos-Renyi} as clear as possible we set $k=1$ in the former. The induced subgraph $H$ is somewhat similar to a random bipartite graph as it has similar properties to $G'\sim B(d,d/n)$: The size of the stables of $H$ is $d$ and edges appear in ($G$ and hence also in) $H$ with uniform probability $d/n.$ The main difference is that edges do not appear independently in $H$, yet the dependence is generally speaking small. The similarity between $H$ and $G'$ is reflected by the fact that a perfect matching exists in both graphs \whp when $d^2/n - \log{d} \tends +\infty.$  

A related question that has been studied more extensively concerns not necessarily bipartite graphs. The models under consideration are $G_d(n)$ (graph chosen uniformly at random from all $d$-regular graphs on $n$ vertices) and $G(n,p)$ (graph on $n$ vertices where edges are chosen independently with probability $p$) where $p=d/n.$ %When the density $d/n$ is small the models are quite different. For instance, for $d=3$ random regular graphs are Hamiltonian \textit{whp}, while a random graph with the same density, is not. However, when the density is larger, random graphs are almost regular \textit{whp}; that is all the degrees are not far from their expected value. 
Two kinds or results have been obtained. On the one hand properties of graphs that hold \whp in $G(n,p)$ have been shown to also hold \whp in $G_d(n)$ \cite{ksvw2001,ksv2002,ksv2007}. On the other hand Kim and Vu have studied the contiguity of both models in~\cite{kv2004}. They conjectured the two models are contiguous when $d\gg \log{n}$ (Sandwich conjecture), but were only able to show a slightly weaker relation between the models when $d\ll n^{1/3}/\log^2{n}$. If their result could be extended to $d$ in the $\sqrt{n \log{n}}$ range (and also to bipartite graphs) it would imply that the induced subgraph $H$ and $G'\sim B(d,d/n)$ are also contiguous, giving a straightforward proof of Theorem~\ref{G[A,B]} as a corollary of Theorem~\ref{Erdos-Renyi}.

The main motivation to study the existence of perfect matchings in induced subgraphs of random
biregular bipartite graphs has to do with commutative graphs and Pl\"unnecke's inequality. A
comprehensive study of the applications of commutative graphs and Pl\"unnecke's inequality can be
found in \cite{Ruzsa2009}. Here we only present the necessary facts that relate commutative graphs
with Theorem~\ref{G[A,ga(y)]}. We begin with the definition.  
\begin{definition}\label{commutative}
A directed layered graph $G$ with vertex set $X_0\cup X_1 \cup \dots \cup X_h$ is called
\emph{commutative} if 
\begin{enumerate}
\item There are edges only between consecutive layers, so that $E(X_i,X_j)=\emptyset$ unless $j=i+1$
for all $0\leq i,j \leq h.$
\item For all $1\leq i \leq h$ and $uv\in E(X_{i-1},X_i)$ there exists a perfect matching from a
subset of $\ga(u)$ to
$\ga(v).$ This condition is called \emph{Pl\"unnecke's upward condition}. 
\item For all $1\leq i \leq h$ and $uv\in E(X_{i},X_{i+1})$ there exists a perfect matching from a
subset of $\gai(v)$ to
$\gai(u).$ This condition is called \emph{Pl\"unnecke's downward condition}.
\end{enumerate}
\end{definition}
%\guillem{Shall we make a comment here on the part \emph{a subset of}? You told me that for some applications you can't have a perfect matching between the neighbourhoods because they can have different sizes, and that for this reason you need to state it like this. But in the article is always considered to be a perfect matching. Perhaps just a sentence like, if $G$ is out-regular and $in$-regular, then blablabla...}
Observe that when $G$ is biregular the perfect matching in Pl\"unnecke's upward (downward) condition
is from the whole $\ga(u)$ to $\ga(v)$ ($\gai(v)$ to $\gai(u)$).  

Pl\"unnecke introduced commutative graphs to study the growth of sumsets \cite{Plunnecke1970,Ruzsa1989,Nathanson1996,Tao-Vu2006}. He was interested in the \emph{magnification ratios} of graphs.
$$
D_i(G) = \min_{\emptyset\neq Z \subseteq X_0} \frac{|\ga^{(i)}(Z)|}{|Z|}\;,
$$
where $1\leq i \leq h$ and $\ga^{(i)}(Z)$ is defined iteratively by $\ga^{(i)}(Z) =
\ga(\ga^{(i-1)}(Z))$. %and $\ga^{(0)}(Z)=Z.$ 
Pl\"unnecke proved a powerful inequality that limits the growth of magnification ratios of commutative graphs.
\begin{theorem}[Pl\"unnecke]\label{Plunnecke}
Let $G$ be a commutative graph. Then the sequence $D_i(G)^{1/i}$ is decreasing.
\end{theorem}   
In \cite{GPPlIn} it was shown that the upper bound for $D_i(G)\leq D_1(G)^i$ given by Theorem~\ref{Plunnecke} is sharp. In particular a commutative graph $G$ that satisfies $D_i(G)=k^i$ for all $1\leq i \leq h$ was
constructed for all $k\in \Q^+$ and $h\in \Z^+$. The extremal examples were biregular commutative graphs whose in and out degrees satisfied $d^+/d^-=k.$ In fact it is easy to check that, in any commutative graph whose degrees satisfy $d^+/d^-=k$, the sequence $D_i(G)^{1/i}$ is constant and equal to $k$. 

We apply Theorem~\ref{G[A,ga(y)]} to give a non-constructive, and probabilistic in nature, proof of
the existence of graphs that are extremal for Pl\"unnecke's inequality, answering a question of Gowers.

We form a layered directed biregular graph by ``placing random biregular bipartite directed graphs on
top of each other.'' This simple construction works when the out-degree is large enough compared to the size of the bottom layer and
the resulting graph is \whp commutative.

\begin{theorem}\label{theorem Commutative}
Let $1 \leq k\in \Q^+,$ $m\in \Z^+ $ be arbitrarily large, $d\in \{2,\dots,m\}$ and $h\in \Z^+$. Suppose that $km, kd\in
\Z^+.$  

Furthermore let $X_0, X_1, \dots , X_h$ be sets with $|X_i| = k^i m.$ For $1\leq i \leq h$ let
$G_i:=G_i[X_{i-1},X_i] \sim G(k,k^{i-1}m,d)$. 

Let $G$ be a graph with vertex set $V(G)=X_0\cup\dots\cup X_h$ and edge set $E(G)= \cup_{i=1}^h
E(G_i).$ Then

\begin{itemize}
\item[(i)] The graph $G$ is not commutative \whp when
$d\leq  \sqrt{\frac{1}{3} k^{h-2}m \log(k m)}.$
\item[(ii)] The graph $G$ is commutative \whp when
$d\geq 3 \sqrt{k^{h-2}m \log(h k^{h+1} m)}.$
\end{itemize}

\end{theorem}

% \begin{theorem}\label{Lower commutative}
% Let $k\in \Q^+,$ $m\in \Z^+ $ be arbitrarily large and $d\in \{2,\dots,n\}$. Suppose that $km,
% kd\in % \Z^+.$  
% 
% Furthermore let $X_0, X_1, \dots , X_h$ be sets and $|X_i| = k^i m.$ For $1\leq i \leq h$ let
% $G_i:=G_i[X_{i-1},X_i]$ be such that $G_i\sim G(k,k^{i-1}m,d)$. 
% 
% Let $G$ be a graph with vertex set $V(G)=X_{h-2}\cup X_{h-1}\cup X_h$ and edge set $E(G)=
% E(G_{h-1})\cup E(G_{h}).$ The graph $G$ is not commutative \whp when
% $$
% d\leq  \sqrt{\frac{1}{3}m k^{h-2} \log(h k^{h-1} m)}\;.
% $$
% % Furthermore let $X_0, X_1, X_2$ be sets and $|X_i| = k^i m$ for $0\leq i \leq 2$. Let
% % $G_1:=G_1[X_0,X_1]$ be such that $G_1\sim G(k,m,d)$ and
% % $G_2:=G_2[X_1,X_2]$ be such that $G_2\sim G(k,km,d)$. 
% % 
% % Let $G$ be a graph with vertex set $V(G)=X_0\cup X_1\cup X_2$ and edge set $E(G)= E(G_1)\cup
% % E(G_2).$ The graph $G$ is not commutative \whp when
% % $$
% % \frac{d^2}{m} - \log(kd) \tends -\infty\;.
% % $$
% \end{theorem}

Observe that the upper bound and the lower bound in Theorem~\ref{theorem Commutative} have the same
asymptotic order. We make no effort to optimize the constants as our method does not lead to matching lower and upper bounds.

%The methods we use in the paper are somewhat different to those  that appear in most of the literature. 
Results on random regular graphs are usually derived using the so-called \emph{configuration (or pairing) model} due to
Bollob\'as~\cite{Bollobas1980} (for a detailed presentation see \cite{Wormald1999}). However, this
model does not give meaningful results when the degree is large. McKay introduced in~\cite{McKay1981} a new way to approach problems in random regular graphs when
the degree is large, based on switching the edges of the graph. This method has been successfully
applied to extend the a lot of results for random regular graphs with large
degree~\cite{McKay-Wormald1991,ksvw2001,cfrr2002,cfr2002,ksv2002,ksv2007}. 

Our strategy is to mirror the proof of Theorem~\ref{Erdos-Renyi} of Erd\H{o}s and R\'enyi. The biggest obstacle is dealing with
dependencies among the edges. We do this by repeatedly using three ingredients: the regularity of the degrees, the symmetry of $\G(k,n,d)$
and the idea of edge switchings.

The existing estimates on the number of biregular bipartite graphs contain error terms, which are
negligible when $d$ is small compared to $n$, but become significant for larger $d$. We will not
need to estimate $|\G(k,n,d)|$ and so we will not be affected by this.

The paper is organised as follows. In the next section we introduce the methods we will use repeatedly throughout the paper. In Section~\ref{Preliminary} we
prove a useful result, whose proof demonstrates how the lack of independence in choosing the edges
can be overcome. In Section~\ref{Thm 3} we prove Theorem~\ref{G[A,ga(y)]} and in Section~\ref{Thm
2} we present the backbone of the proof of Theorem~\ref{G[A,B]}. Finally in
Section~\ref{Commutative} we prove Theorem~\ref{theorem Commutative}. 

\begin{acknowledgement}
Both authors would like to thank Oriol Serra for his help and support which were instrumental for the completion of the project. The second author would like to thank Tim Gowers for suggesting using random graphs to construct extremal examples for Pl\"unnecke's inequality and for sharing his insight. He would also like to thank Ben Green and Peter Keevash for helpful suggestions.
\end{acknowledgement}

\textbf{Notation.} We conclude the introduction with a quick recap of standard notation we will use
throughout the paper. Then \emph{out-degree} of a vertex $v$ is $d^+(v) = |\ga(v)|$ and the
\emph{minimum out-degree} of a directed graph $G$ is $\delta^+(G) = \min\{d^+(v) : v\in V(G)\}.$ The
\emph{in-degree} of a vertex $v$ is similarly defined by $d^-(v)=|\gai(v)|$ and so is the \emph{minimum
in-degree} $\delta^-(G)$ of a directed graph $G$. The \emph{minimum degree} of a directed graph $G$
is $\delta(G) = \min\{\delta^-(G), \delta^+(G)\}.$

For a two functions $f, g$ we write $f(n) = O(g(n))$ if $|f(n)|\leq C |g(n)|$ for some absolute
constant $C$ and $f(n)= \Theta(g(n))$ if $f(n) = O(g(n))$ and $g(n) = O(f(n)).$ We also write
$f(n) = o(g(n))$ to mean that $\lim_{n\tends +\infty} f(n)/g(n) =0.$ In particular we use
$f(n)=o(1)$ if $\lim_{n\tends +\infty} f(n) = 0.$

%\comment{THE OLD INTRO TEXT IS RIGHT AFTER THE $\backslash$end{document}}

\section{Symmetry and switching in biregular bipartite graphs} \label{Symmetry and switching}

%The lack of independence in choosing the edges is the main obstruction in obtaining an Erd\H{o}s--R\'enyi result in random biregular graphs as it is generally speaking hard to obtain exact expressions for the probability that an event occurs. 

The first result we prove illustrates how the regularity of the degrees and the symmetry of biregular bipartite graphs will be used in the paper.

\begin{lemma}\label{Expectation of intersection}
Let $k,n,d$ be like in the statement of Theorem~\ref{G[A,B]}. Suppose that $G\sim G(k,n,d).$

(i) Let $y,y'\in Y.$ Then
$$
\E(|\ga(y)\cap \ga(y')|) = \frac{k d (d-1)}{n-1}\;.
$$
(ii) Let $y\in Y$ and $B\subseteq Z.$ Then  
$$
\E(|\ga(y)\cap B|) = \frac{d |B|}{n}\;.
$$
\end{lemma} 
\begin{proof}
First note that $\ga(y)$ is chosen uniformly  at random from all $(kd)$-element subsets of $Z.$ Without loss of generality we can therefore assume that it is fixed and equal to a set $S$. Next we observe that  
\begin{eqnarray*}
\E(|\ga(y)\cap \ga(y')|)  =  \sum_{z\in S} \Pr(z\in \ga(y')) \;.
\end{eqnarray*}
The probability $\Pr(z\in \ga(y'))$ is equal for all $z\in S.$ To see why take $z_0, z_1\in S$ and observe that there exists a bijection $\theta$ from
$$
\G_{z_0} =\{G\in \G{(k,n,d)} : \ga(y)= S \wedge y'z_0\in E(G) \}
$$
to
$$
\G_{z_1} =\{G\in \G{(k,n,d)} : \ga(y)= S  \wedge y'z_1\in E(G) \}\;.
$$
The bijection $\theta$ maps $z_0$ to $z_1$ and vice versa and restricts to the identity on
$V(G)\setminus \{z_0, z_1\}$. So 
\begin{eqnarray*}
\E(|\ga(y)\cap \ga(y')|)  =kd  \Pr(z_0\in \ga(y'))\;. 
\end{eqnarray*}
On the other hand
\begin{eqnarray*}
d-1 & = & \E(|\gai(z_0) \setminus \{y\}|) \\
       & = & \sum_{v\in Y\setminus \{y\}} \Pr(z_0\in \ga(v))\\
       & = & (n-1) \Pr(z_0\in \ga(y'))\;.
\end{eqnarray*}
The third identity following from the symmetry of random biregular bipartite graphs. The first conclusion follows. The second can be proved similarly.
\end{proof} 

The arguments in the above proof are not sufficient when dealing with more complicated events. To deal with such events we will employ elementary counting arguments that involve switchings.   

\begin{definition}\label{Switching}
Let  $a,b\in Y$ and $c,d\in Z$ such that $ac,bd\in E(G)$ and $ad,bc\notin E(G).$ The
\emph{$\{ac,bd\}$-switching} of $G$ is the graph $H$ with the same set of vertices as $G$ and
$E(H)=E(G)\cup \{ad,bc\}\setminus\{ac,bd\}$. 
\end{definition}
Figure~\ref{fig:sw1} offers an illustration of this natural operation. Observe that if $G$ is
biregular bipartite, then so is $H$; and that if $H$ is the $\{ac,bd\}$-switching of $G$, then $G$
is the $\{ad,bc\}$-switching of $H$.
\begin{figure}[ht]
 \begin{center}
 \includegraphics[width=0.7\textwidth]{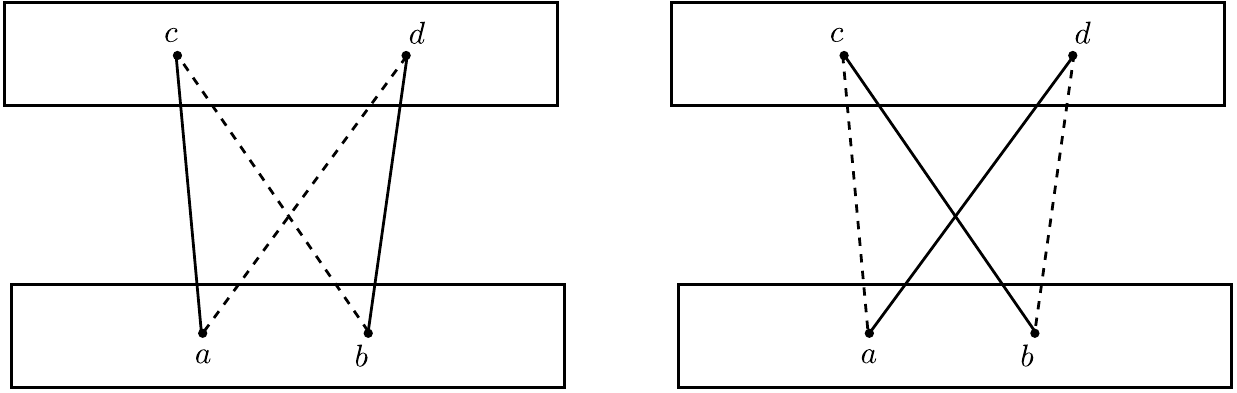}
 \end{center}
 \caption{A graph $G$ and its $\{ac,bd\}$-switching $H$. Solid lines represent edges and dashed
lines missing edges.}
 \label{fig:sw1}
\end{figure}
Switchings between graphs were first used by McKay in~\cite{McKay1981} to obtain bounds on the probability that a
fixed graph appears as a subgraph of a random regular graph. 
McKay~\cite{McKay1985} used the same technique to extend the range of $d$ in the enumeration of
regular graphs to $d=o(n^{1/3})$. McKay and Wormald in~\cite{McKay-Wormald1991} improved that range
to $d=o(\sqrt{n})$ by introducing a new type of switching. Switching is moreover
useful in proving that \whp regular graphs are expanders~\cite{Friedman1991} or in counting the
number of spanning trees subject to an asymptotic condition on the number of cycles~\cite{McKay1983}.

As mentioned in the introduction switching has more recently been used to study various properties of random regular graphs \cite{ksvw2001,cfrr2002,cfr2002,ksv2002,ksv2007}. We will use it in a similar fashion to compare the sizes of two families of biregular bipartite graphs, say $\G_1$ and $\G_2.$ We will do this counting in two ways the number of switchings between the two families. In other words we will double count the number of ordered pairs $(G_1,G_2) \in \G_1 \times \G_2$ where $G_1$ is a switching of $G_2,$ which is equivalent to $G_2$ being a switching of $G_1.$  

\section{Preliminary results} \label{Preliminary}

The key to most of the calculations leading to the proof of Theorem~\ref{G[A,ga(y)]} is having a
good upper bound on the probability that there are no edges between two sets $S\subseteq A$ and
$T\subseteq \ga(y)$, for some $y\in A.$ As $y$ is joined to all vertices in $\ga(y)$ we assume that
$S\subseteq
A\setminus\{y\}.$ The main result of this section is the following. 

\begin{proposition} \label{No edge from S to T}
Let $Y,Z, G $ and $y$ be like in the statement of Theorem~\ref{G[A,ga(y)]}. Suppose that $T\subseteq
\Gamma(y), z_1\in T$ and $S\subseteq Y\setminus\{y\}$, where $|S|+d\leq n$. Then the probability
that there are no edges from $S$ to $T$ is bounded above by:
\begin{eqnarray*}
\Pr(\Gamma(S)\cap T = \emptyset) &\leq& \Pr(\gai(z_1)\cap S = \emptyset)^{|T|} \\
						      &  =   & \left(1-\frac{|S|}{n-1}\right)^{|T|}
\left(1-\frac{|S|}{n-2}\right)^{|T|} \dots \left(1-\frac{|S|}{n-d+1}\right)^{|T|} \\
						      &\leq&
\left(1-\frac{d-1}{n-1}\right)^{|S|\,|T|} \\
						      &\leq& (1+o(1))\, \exp{\left(-
\frac{d\,|S|\,|T|}{n}\right)}\;.
\end{eqnarray*}
\end{proposition}
Before giving the proof we quickly present a heuristic explanation for the crucial first inequality. For simplicity we take $T=\{z_1,z_2\}$. Suppose for a moment that the neighbourhoods of vertices
in $S$ were chosen independently. Then we would have 
$$
\Pr(\Gamma(S)\cap T = \emptyset) = \Pr(\gai(z_1)\cap S = \emptyset)  \Pr(\gai(z_2)\cap S =
\emptyset) =  \Pr(\gai(z_1)\cap S = \emptyset)^2 \;.
$$  
As we do not have independence we have to instead use conditional probabilities:
$$
\Pr(\Gamma(S)\cap T = \emptyset) = \Pr(\gai(z_1)\cap S = \emptyset)  \Pr(\gai(z_2)\cap S = \emptyset
\mid \gai(z_1)\cap S = \emptyset)\;.
$$  
Conditioning on the event $\gai(z_1)\subseteq Y\setminus S$ has an effect on $ Y\setminus S:$ the
vertices in $\gai(z_1)$ have one of the $kd$ edges coming out of them ``taken up'' by $z_1.$ One
expects that this makes $\gai(z_2)$ less likely to include them and consequently that 
$$
\Pr(\gai(z_2)\cap S = \emptyset \mid \gai(z_1)\cap S = \emptyset) \leq \Pr(\gai(z_1)\cap S =
\emptyset) \;.
$$  

Proving this type of upper bound for conditional probabilities is the main task lying ahead. %In addition to the degree regularity and symmetry of the graph $G$, we will rely on the switching operation described in the previous section.

\begin{proof}[Proof of Proposition~\ref{No edge from S to T}]
The second inequality can be proved by induction on $d$ and the third is standard, so we only prove
the first inequality and the expression for $ \Pr(\gai(z_1)\cap S = \emptyset).$ We let $s=|S|$, 
$T=\{z_1,\dots,z_t\}$ and proceed by induction on $t.$ 

When $t=1$ we have
 \begin{eqnarray}\label{ground case for e(S,T)=0}
 \Pr(\gai(z_1) \cap S = \emptyset ) = \frac{\binom{n-1-s}{d-1}}{\binom{n-1}{d-1}} = \left(1 -
\frac{s}{n-1}\right)  \left(1 - \frac{s}{n-2}\right) \dots  \left(1 - \frac{s}{n-d+1}\right)\;,
 \end{eqnarray}
as $\gai(z_1)\setminus \{y\}$ is uniformly distributed over all $(d-1)$-element subsets of
$Y\setminus \{y\}.$ Another way to interpret this identity is by ordering the edges coming in $z_1.$
Without loss of generality we can assume that the first edge is $yz_1.$ The probability the second
edge coming in $z_1$ does not originate from $S$ is $1-s/(n-1).$ The probability the third edge
coming in $z_1$ does not originate from $S$, given that the second does not, is $1-s/(n-2)$ and so
on. 

For the inductive step let us write $T'=\{z_1,\dots, z_{t-1}\}.$ As 
$$
\Pr(\gai(T) \cap S = \emptyset) =  \Pr(\gai(T') \cap S = \emptyset) \, \Pr(\gai(z_t) \cap S =
\emptyset \mid \gai(T') \cap S = \emptyset) \;,
$$
it is enough for our purpose to establish that
\begin{eqnarray} \label{inductive step for e(S,T)=0}
\Pr(\gai(z_t) \cap S = \emptyset \mid \gai(T') \cap S = \emptyset)  &\leq& \Pr(\gai(z_t) \cap S =
\emptyset) \\
&  =   & \Pr(\gai(z_1) \cap S = \emptyset)\;. \nonumber
\end{eqnarray}
The last equality follows from the symmetry properties of biregular bipartite graphs. The remainder
of the proof is dedicated to proving \eqref{inductive step for e(S,T)=0}. The
strategy is to order the edges ending in $z_t$ and successively estimate the probability that each
does not originate from $S.$ This will be done in a number of lemmata. 

We need to keep track of the first $j$ edges ending in $z_t.$ To achieve this we denote by
$y=y_1,\dots,y_d$ the elements of $\gai(z_t)$ and, for $1\leq j \leq d,$ we set $F_j=
\{y_1,\dots, y_j\}.$  %Observe that $F_1=\{y\}$ and $F_d=\gai(z_t).$

The key is to prove the following intuitively clear observation. Suppose that $\gai(T')$ and $F_j$ are disjoint from $S$. Then for any $u\in S$ and any $v\in Y\setminus F_j$ the probability that $y_{j+1}=u$ is no smaller than the probability that $y_{j+1}=v$. We prove the statement in a number of steps. Initially we condition on $F_j$
and $\gai(T')$.

\begin{lemma}\label{u vs notin gai(T')}
Let  $1\leq j\leq d-1$ be an integer and $u\in S.$ Suppose $J\subseteq Y\setminus S$ is a set of
size $j$ that contains $y$, $W\subset Y\setminus S$ is another subset of $Y$ that is disjoint from
$S$ and $v\notin W\cup J.$ Then
$$
\Pr(y_{j+1} =v \mid F_j =J \wedge \gai(T')=W)=\Pr(y_{j+1} = u \mid F_j =J \wedge \gai(T') = W)\;.
$$
\end{lemma}
\begin{proof}
The statement follows from the symmetry properties of random biregular bipartite graphs:
interchanging $u$ and
$v$ does not affect the events $\{\gai(T') = W\}$ nor $\{F_j=J\}.$  
\end{proof}

\begin{lemma}\label{u vs in gai(T')}
Let  $1\leq j\leq d-1$ be an integer and $u\in S.$ Suppose $J\subseteq Y\setminus S$ is a set of
size $j$ that contains $y$, $W\subset Y\setminus S$ is another subset of $Y$ that is disjoint from $S$
and $v\in W\setminus J.$ Then
$$
\Pr(y_{j+1} =v \mid F_j =J \wedge \gai(T') = W) \leq \Pr(y_{j+1} = u \mid F_j =J \wedge
\gai(T') = W)\;.
$$
\end{lemma}

\begin{proof}

We first observe that, say, 
$$
\Pr(y_{j+1} = v \mid F_j =J \wedge \gai(T') = W) = \frac{\Pr(v \in \gai(z_t) \mid
F_j =J
\wedge \gai(T') = W)}{d-j}\;
$$
since $v$ could be any of the $d-j$ remaining vertices in $\gai(z_t)\setminus J$ with uniform
probability.  So the statement of the lemma is equivalent to
$$
\Pr(v \in\gai(z_t) \mid F_j =J \wedge \gai(T') = W) \leq \Pr(u \in \gai(z_t)
\mid F_j =J \wedge \gai(T') = W)\;.
$$ 
Subtracting the probability $\Pr(\{u, v\} \subseteq \gai(z_t) \mid F_j=J \wedge \gai(T') = W)$ from both sides of the inequality leaves us with having to prove that
$$
\Pr(v \in \gai(z_t) \wedge u\notin \gai(z_t) \mid F_j=J \wedge \gai(T') = W) 
$$
is at most
$$
\Pr(u \in  \gai(z_t) \wedge v\notin \gai(z_t) \mid F_{j}=J \wedge \gai(T') = W) \;.
$$
To this end we define two families of graphs:
$$
\G_{v} = \{G \in \G(k,n,d) : v \in  \gai(z_t) \wedge u\notin \gai(z_t) \wedge F_{j}=J \wedge
\gai(T') = W\}
$$
and 
$$
\G_{u} = \{G \in \G(k,n,d) : u \in  \gai(z_t) \wedge v\notin \gai(z_t) \wedge F_{j}=J \wedge
\gai(T') = W \} \; ;
$$ 
and we establish that $|\G_v| \leq |\G_{u}|.$ 

For this purpose it is advantageous to know the size of the intersection $\ga(u)\cap\ga(v).$ In
$\G_v$ the
intersection is at most $kd-2$ as both $z_t$ and an element of
$T'$ lie in $\ga(v)\setminus \ga(u)$. In $\G_u$ the intersection is at most $kd-1$ as $z_t\in
\ga(u)\setminus\ga(v).$ For $0\leq i \leq kd-2$ we define new families of graphs
$$
\G_{v,i} = \{G \in \G(k,n,d) : v \in  \gai(z_t) \wedge u\notin \gai(z_t)  \wedge |\ga(u)\cap \ga(v)|
= i  \wedge F_{j}=J \wedge \gai(T') = W \}
$$
and 
$$
\G_{u,i} = \{G \in \G(k,n,d) : u \in  \gai(z_t) \wedge v\notin \gai(z_t) \wedge |\ga(u)\cap \ga(v)|
= i
\wedge F_{j}=J \wedge \gai(T') = W \}\;.
$$
To finish the proof it is enough to show that for all $0\leq i \leq
kd-2$ we have $|\G_{v,i}| \leq |\G_{u,i}|.$ We establish this by counting in two ways $N_i$, the
number of switchings (introduced below the proof of Lemma~\ref{Expectation of intersection} in
p.~\pageref{Switching}) between $\G_{v,i}$ and $\G_{u,i}.$ I.e. we double count the number of
pairs $(G_v,G_{u}) \in
\G_{v,i}\times\G_{u,i}$ where $G_u$ is a switching of $G_{v}$ or equivalently $G_{v}$ is a switching
of $G_u.$

\begin{figure}[ht]
 \begin{center}
 \includegraphics[width=\textwidth]{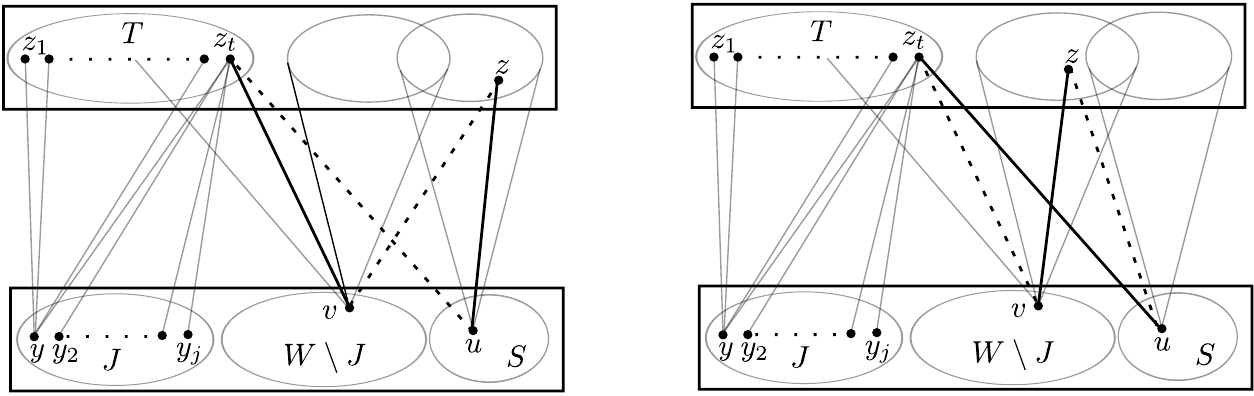}
 \end{center}
 \caption{A graph $G_v\in\G_v$ and its switching $G_u\in\G_u$. Solid lines represent edges and
dashed lines missing edges.}
 \label{fig:sw2}
\end{figure}

We pick $G_v\in \G_{v,i}$ and let $z\in \ga(u) \setminus \ga(v).$  Applying the
$\{vz_t, uz\}$-switching to $G_v$ gives a graph $G_u\in\G_{u,i},$ as the switching does not
affect neither $|\ga(u) \cap \ga(v)|$ nor the event $\{F_j=J \wedge \gai(T') = W\}$, and disconnects $v$
from $z_t$ by connecting it to $u$~(see Figure~\ref{fig:sw2}). 
There are $kd-i$ vertices in $\ga(u) \setminus \ga(v)$ and so 
$$
N_i = (kd-i) |\G_{v,i}| \;.
$$
Next we pick $G_u\in \G_{u,i}$ and let $z\in \ga(v)\setminus (\ga(u) \cup T')$. Just like
above, applying the $\{uz_t, vz\}$-switching to $G_u$ gives a graph $G_v\in\G_{v,i}$. This time
however there are at most $(kd-1-i)$ vertices in $\ga(v)\setminus(\ga(u) \cup T'),$ as $v\in \gai
(T')$. Thus,
$$
N_i \leq (kd-i-1) |\G_{u,i}|\;.
$$ 
Comparing the lower and upper bounds for $N_i$ gives $|\G_{v,i}| \leq |\G_{u,i}|$, for $0\leq i \leq
kd-2,$ and so
$$
|\G_v| = \sum_{i=0}^{kd-2} |\G_{v,i}| \leq \sum_{i=0}^{kd-2} |\G_{u,i}| \leq |\G_{u}| 
$$
as required.
\end{proof}

We resume the proof of Proposition~\ref{No edge from S to T} by noting that
\begin{eqnarray*}%\label{total probability}
1 = \sum_{v\in Y\setminus J} \Pr(y_{j+1} =v\mid F_j =J \wedge \gai(T') = W)\;,
\end{eqnarray*}
where $J$ and $W$ are taken to be subsets of $Y\setminus S$ as in the statement of the preceding
lemma. The two preceding lemmata above imply that  
$$ 
\Pr(y_{j+1} = v \mid F_j =J \wedge \gai(T') = W) \leq \Pr(y_{j+1} = u \mid F_j =J \wedge
\gai(T') = W)\;,
$$ 
for any $v\in Y\setminus J$ and some fixed $u\in S.$ Substituting in the identity above and using the symmetry of the vertices in $S$ leads to
$$
1\leq (n-j) \Pr(y_{j+1} = u \mid F_j =J \wedge \gai(T') = W)= (n-j)\frac{\Pr(y_{j+1} \in S \mid F_j
=J \wedge \gai(T') = W)}{s}\;,
$$
which in turn implies
$$
\Pr(y_{j+1} \notin S \mid F_j =J \wedge \gai(T') = W) = 1- \Pr(y_{j+1} \in S \mid F_j =J \wedge
\gai(T') = W) \leq 1-\frac{s}{n-j}\;.
$$
We take advantage of the fact that the upper bound is independent of $J$ and $W$ to deduce an upper
bound
for the probability $\Pr(y_{j+1} \notin S  \mid (F_j \cup \gai(T'))\cap S = \emptyset ).$

\begin{corollary}
Let $1\leq j \leq d-1$. Then $\Pr(y_{j+1} \notin S\mid (F_j \cup\gai(T') ) \cap S = \emptyset )  
\leq  1-s/(n-j).$
\end{corollary}
\begin{proof}
The probability
$$
\Pr(y_{j+1} \notin S \mid (F_j \cup \gai(T') ) \cap S = \emptyset )  
$$
equals
$$
\sum_{J,W}  \Pr(y_{j+1} \notin S \mid F_j =J \wedge \gai(T') = W) \, \Pr(F_j =J \wedge\gai(T') = W)
\;, 
$$
where the sums are over all $j$-element subsets $J$ of $Y\setminus S$ and all subsets $W\subset
Y\setminus S.$ This in turn is at most
$$
\left(1-\frac{s}{n-j}\right) \sum_{J,W} \Pr(F_j =J \wedge \gai(T')=W)    =    1-\frac{s}{n-j} \;.
\qedhere
$$
\end{proof}

We can finally deduce \eqref{inductive step for e(S,T)=0}. Recall that
$\gai(z_t)=\{y_1,\dots,y_d\}$, where $y_1=y$.  
\begin{eqnarray*}
\Pr(\gai(z_t) \cap S =\emptyset \mid  \gai(T') \cap S=\emptyset)   &  =  &  \prod_{j=1}^{d-1} 
\Pr(y_{j+1} \notin S \mid (F_j \cup \gai(T') ) \cap S = \emptyset ) \\	         
&\leq&  \prod_{j=1}^{d-1}  \left(1-\frac{s}{n-j}\right)\\
& =  &  \Pr(\gai(z_1) \cap S = \emptyset)\;.
\end{eqnarray*}
Equation \eqref{ground case for e(S,T)=0} was used for the last equality. This finishes the
inductive step and concludes the proof of Proposition~\ref{No edge from S to T}.
\end{proof}  
%\giorgis{I sort of prefer to condition on $F_j=J$ as I wouldn't want to write down things like
% $\Pr(F_{j+1} = v \cup F_j).$ sadly it makes an already long proof even longer}\guillem{I labeled
% the vertices from $\gai(z_t)$. I think that now it's clearer.}

A special case that will be of particular importance in the next section is when $S=A\setminus\{y\}$
and $T$ is a singleton. For this case we would like to have not only an upper bound, but also an asymptotic expression for the
probability that there are no edges between the two sets. %for $z_1\in \ga(y)$ we will need an asymptotic expression for the
% probability that $\gai(z_1)\cap A =\{y\}.$ 

\begin{lemma}\label{Pr(A^-_z)}
Let $Y,A, Z, G $ and $y$ be like in the statement of Theorem~\ref{G[A,ga(y)]}. Suppose that $
z_1\in\ga(y).$ Then %the probability that there are no edges from $A\setminus\{y\}$ to $z_1$ is asymptotically equal to
\begin{eqnarray*}
 \Pr(\gai(z_1)\cap A = \{y\}) = (1+o(1)) \exp{\left(-\frac{kd^2}{n}\right)}\;,
\end{eqnarray*}
provided only that $d=o(n^{2/3}).$
\end{lemma}
 \begin{proof}
Setting $T=\{z_1\}$ and $S= A\setminus \{y\}$ in Proposition~\ref{No edge from S to T} gives
 $$
 \Pr(\gai(z_1)\cap A = \{y\}) \leq (1+o(1))  \exp\left(- \frac{k d^2}{n}\right)\;.
 $$
To get a lower bound we first note that Proposition~\ref{No edge from S to T} gives 
$$
 \Pr(\gai(z_1)\cap A = \{y\}) =   \left(1-\frac{kd-1}{n-1}\right) \left(1-\frac{kd-1}{n-2}\right)
\dots \left(1-\frac{kd-1}{n-d+1}\right) \geq \left(1- \frac{kd-1}{n-d}\right)^{d-1}\;.
$$
We make use of the inequality $(1-x)\geq (1+O(x^{2}))\exp{(-x-x^2)}$ which holds when $x\tends
0$. 
%\comment{The derivative of
%$F(x)=(1-x)-\exp{(-x-x^2)}$ equals $\exp{(-x-x^2)} (1+2x) -1$, which is positive as $1+2x\geq
%\exp{(x+x^2)}$ for sufficiently small $x$. Alternatively, 
%$(1-x)=\exp\left(\ln(1-x)\right)=\exp\left(-x-(1+o(1))\frac{x^2}{2}
%)\right)\geq\exp\left(-x-(1+o(1))x^2)\right) $.} 
When $d=o(n^{2/3})$ the ratio $\frac{(kd-1)^2(d-1)}{(n-d)^2} = o(1)$ and so
\begin{eqnarray*}
\Pr(\gai(z_1)\cap A = \{y\}) &\geq&
\left(1+O\left(\frac{(kd-1)^2}{(n-d)^2}\right)\right)^{d-1}
\exp{\left(-\frac{kd-1}{n-d}-\frac{(kd-1)^2}{(n-d)^2}\right)}^{d-1} \\
&  = 
&(1+o(1))\exp{\left(-\frac{(d-1)(kd-1)}{n}\right)
} \exp{\left(-\frac{(kd-1)^2}{(n-d)^2} -\frac{d(kd-1)}{n(n-d)}\right)}^{d-1} \\
&  =   & (1+o(1)) \exp{\left(-\frac{k d^2}{n}\right)}\;. \qedhere
\end{eqnarray*}
\end{proof}

 \section{Proof of Theorem~\ref{G[A,ga(y)]}}\label{Thm 3}
 
When $d=o(\sqrt{n})$ it is straightforward to show there is no matching in $H$ \textit{whp}. Take $y' \in
A\setminus\{y\}.$ By Lemma~\ref{Expectation of intersection} we know that the
expected value $\E(|\ga(y)\cap \ga(y')|) =o (1).$ Thus the probability $\Pr(\ga(y) \cap \ga(y') =
\emptyset) = 1-o(1)$ and consequently there is no matching in $H$ with
probability $1-o(1).$ For larger values of $d$ this simple argument does not work.
 
We follow Erd\H{o}s and R\'enyi in relating the event of finding a perfect matching in $H$ to the
event that the minimum degree of the induced subgraph is 1. As $k$ is not necessarily 1
and $y$ lies in the bottom layer there is no symmetry between the top and bottom layers. To deal
with this is a technical difficulty we will to consider $\delta^+(H)$ and $\delta^-(H)$ separately.
The proof of Theorem~\ref{G[A,ga(y)]} is broken down to four steps.

The first is to obtain a qualitative description of the range of $d$ for which $\delta^-(H)=1$
\textit{whp}.
Note that $\delta^-(H)$ cannot be zero since $y\in A$.
 \begin{lemma}\label{delta^-(H)}
 Let $H$ be the graph introduced in Theorem~\ref{G[A,ga(y)]} and 
 $$
 c = \frac{kd^2}{n} - \log(kd)\;.
 $$
 Then
 \begin{enumerate}
\item[(i)] $\Pr(\delta^-(H)=1)= 1 -o(1)$ when  $c \rightarrow - \infty$ or when $d$ is a constant.
\item[(ii)] $\Pr(\delta^-(H)>1)=1-o(1)$ when $c \rightarrow +\infty.$ 
\end{enumerate}
Furthermore there is no perfect matching in $H$ \whp when $c\rightarrow -\infty.$
 \end{lemma}
\begin{proof}
We estimate the expectation and variance of the number of vertices $z\in \Gamma(y)$
that satisfy $\gai(z)\cap A = \{y\}$. So for $z\in \ga(y)$ we define the event 
\begin{eqnarray*}%\label{A_z}
A^-_z =\{\gai(z)\cap A = \{y\} \}\;;
\end{eqnarray*}
and the random variable
\begin{eqnarray*}%\label{A-}
A^- = \sum_{z\in \ga(y)} 1_{A^-_{z}}\;.
\end{eqnarray*}

The linearity of expectation and the symmetry of biregular bipartite graphs gives
\begin{eqnarray}\label{E(A)}
\E(A^-) = kd \Pr(A^-_{z_1})\qquad \mbox{ for any $z_1\in \ga(y)$}\;.
\end{eqnarray}
Setting $T=\{z_1\}$ and $S=A\setminus \{y\}$ in Proposition~\ref{No edge from S to T} yields
$$
\E(A^-) \leq  (1+o(1)) kd\, \exp\left(-\frac{d (kd-1)}{n}\right) =O\left( kd\,
\exp\left(-\frac{kd^2}{n}\right)\right) = O(e^{-c})\;.
$$
When $c\rightarrow + \infty$ the expectation is $\E(A^-) = o(1)$ and so $\Pr(A^->0)\leq
\E(A^-)=o(1)$ and conclusion $(ii)$ follows.  

When $c\leq 0$ we certainly have that $d=o(n^{2/3})$ and so Lemma~\ref{Pr(A^-_z)} gives
$$
\E(A^-)=\sum_{z\in \ga(y)}  \Pr(\gai(z)\cap A = \{y\}) = (1+o(1)) kd \,
\exp\left(-\frac{kd^2}{n}\right) = (1+o(1)) e^{-c}\;.
$$
Assuming furthermore that $c\rightarrow - \infty$ gives $\E(A^-)\rightarrow+\infty.$ To be able to
say something about the probability $\Pr(A^->0)$ we need to control the variance of $A^-.$
\begin{eqnarray*}
\var(A^-)     &  =   & \mathbb{E}((A^-)^2) -(\mathbb{E}(A^-))^2 \\
                        &  =   & \sum_{z\in \ga(y)}\sum_{ z'\in \ga(y)} \Pr(A^-_{z}\wedge
A^-_{z'}) - \left(\sum_{z\in \Gamma(y)} \Pr(A^-_{z})\right)^2\\
                        &  =   & \sum_{z\in \ga(y)}\sum_{z'\neq z} \Pr(A^-_{z}\wedge
A^-_{z'}) + \sum_{z\in \ga(y)} \Pr(A^-_{z})  - \left(\sum_{z\in \Gamma(y)} \Pr(A^-_{z})\right)^2\;.
\end{eqnarray*}
Setting $T=\{z,z'\}$ and $S = A\setminus\{y\}$ in Proposition~\ref{No edge from S to T} gives $ \Pr(A^-_{z}\wedge A^-_{z'}) \leq \Pr(A^-_z)^2$. Consequently
\begin{eqnarray*}
\var(A^-)   \leq \sum_{z\in \ga(y)} \Pr(A^-_{z}) (1-\Pr(A^-_{z}))\leq \E(A^-) \;.
\end{eqnarray*}

We can now finish off the proof of the Proposition~\ref{delta^-(H)} by applying Chebyshev's
inequality.
\begin{lemma}[Chebyshev's inequality]\label{Chebyshev}
Let $X$ be a non-negative random variable with expected value $\mu$ and
non-zero variance $\sigma^2.$ Then for any $x\in \mathbb{R}^+$
$$
\Pr(|X-\mu| \geq x \sigma ) \leq \frac{1}{x^2}\;.
$$
\end{lemma}  
Applying the inequality to $A^-$ gives 
$$
\Pr(\delta^-(H)>1) = \Pr(A^-=0) \leq \Pr(|A^- -\mathbb{E}(A^-)| \geq \mathbb{E}(A^-) ) \leq\frac{1}{\E(A^-)} = o(1) 
$$
when $c\tends -\infty,$ implying conclusion $(i)$.

For the final conclusion we observe that there can be no matching in $H$ when $A^-\geq 2$ as there
would then exist two vertices in $\ga(y)$ that are only joined to $y.$ A second application of
Chebyshev's inequality gives that $\Pr(A^-\geq 2) = 1-o(1)$ when $c\tends -\infty.$ 
\end{proof}

We have proved the first statement in Theorem~\ref{G[A,ga(y)]}. The second statement is trickier.

The second step in the proof of Theorem~\ref{G[A,ga(y)]} is to show that $\delta^+(H) > 0$ \whp when $c\tends +\infty.$

 \begin{lemma}\label{delta^+(H)}
 Let $H$ be the graph introduced in Theorem~\ref{G[A,ga(y)]} and 
 $$
 c = \frac{kd^2}{n} - \log(kd)\;.
 $$
 Then
$$
\Pr(\delta^+(H)=0)=o(1)
$$ 
when $c\tends +\infty.$
\end{lemma}
\begin{proof}
We estimate the expected number of vertices $y'\in A\setminus\{y\}$ that satisfy $\ga(y')\cap
\ga(y) = \emptyset$. So for $y'\in A\setminus\{y\}$ we define the event 
\begin{eqnarray}\label{A^+_y'}
A^+_{y'} =\{\ga(y')\cap \ga(y) = \emptyset \}\;;
\end{eqnarray}
and the random variable
\begin{eqnarray*}%\label{A+}
A^+ = \sum_{y'\in A\setminus \{y\}} 1_{A^+_{y'}}\;.
\end{eqnarray*}

The linearity of expectation and the symmetry of biregular bipartite graphs gives
\begin{eqnarray*}
\E(A^+) = (kd-1)  \Pr(A^+_{y'})\qquad \mbox{ for any $y'\in A\setminus\{y\}$}\;.
\end{eqnarray*}
Setting $S=\{y'\}$ and $T=\ga(y)$ in Proposition~\ref{No edge from S to T} yields
$$
\E(A^+) = O\left( kd\, \exp\left(-\frac{kd^2}{n}\right)\right) = O(e^{-c})\;.
$$
When $c\rightarrow + \infty$ the expectation $\E(A^+) = o(1)$ and so the probability
$\Pr(\delta^+(H)=0)=\Pr(A^+>0)=o(1).$ 
\end{proof}
 
To prove the existence of a perfect matching in $H$ we will rely on the Frobenius-K\"onig theorem
(see e.g.~\cite[Theorem 1.7.1]{Marcus-Minc1992}), which is equivalent to the well known Hall's
theorem.
\begin{lemma}[The Frobenius-K\"onig theorem]\label{FK}
Let $A$ and $B$ be two sets of equal size. Suppose that $H$ is a bipartite graph with vertex set
$(A,B).$ Then $H$ has no perfect matching if and only if there are non-empty sets $S\subseteq A$ 
and $ T\subseteq B$ such that $|S|+|T|= |A|+1$ and $\ga(S)\cap T =
\emptyset.$ 
\end{lemma}
A pair $(S,T)$ is called \emph{problematic} if $|S|+|T|= |A|+1$ and $\ga(S)\cap T =
\emptyset.$
We show that the probability that a problematic pair $(S,T)$ exists in $H$ is $o(1)$ when $c\tends
+\infty$. For technical reasons we need to distinguish between two ranges for $c.$ The third step in the proof of Theorem~\ref{G[A,ga(y)]} is to deal
with the case when $c$ is larger than a constant multiple of $\log(kd).$

 \begin{proposition}\label{Large c}
 Let $H$ be the graph introduced in Theorem~\ref{G[A,ga(y)]} and 
 $$
 c = \frac{kd^2}{n} - \log(kd)\;.
 $$
Suppose that $c\geq 5 \log{(kd)}.$ Then 
 $$
 \Pr(\mbox{There is no perfect matching in $H$}) = O\left( k^2d^2 \exp\left(-\frac{kd^2}{2n}\right)
\right)\;.
 $$
In particular there is a perfect matching in $H$ \textit{whp}.
\end{proposition}
\begin{proof}
It follows by Lemma~\ref{FK} that no perfect matching exists in $H$ if and only if there are
non-empty sets
$S\subseteq A\setminus\{y\}$ and $T\subseteq \ga(y)$ such that $|S|+|T| = kd+1$ and $\ga(S)\cap T =
\emptyset.$ So we want to bound from above the probability a problematic pair of sets
$(S,T)$ exists.
 
For a pair of fixed sets $(S,T)$ where $|S|=j,$ Proposition~\ref{No edge from S to T} gives
\begin{eqnarray*}%\label{E_0}
\Pr(\ga(S)\cap T = \emptyset) \leq (1+o(1)) \exp{\left(-\frac{d |S| |T|}{n}\right)} =  (1+o(1))
\exp\left(-\frac{d j (kd+1-j)}{n}\right)\;. 
\end{eqnarray*}
For a given $j$ there are at most $\tbinom{kd-1}{j} \tbinom{kd}{kd+1-j} \leq \tbinom{kd}{j}
\tbinom{kd}{j-1}$ possible problematic pairs of sets $(S,T).$ Applying a union bound gives
\begin{eqnarray*}
\Pr(\mbox{There is no perfect matching in $H$})    & = & O\left( \sum_{j=1}^{kd} \binom{kd}{j}
\binom{kd}{j-1} \exp\left(-\frac{j (kd-j+1) d}{n} \right) \right) \;.
\end{eqnarray*}
Changing $j$ to $kd-j+1$ does not affect the summand, so 
\begin{eqnarray*}
\Pr(\mbox{There is no perfect matching in $H$})   & = & O\left( \sum_{1\leq j \leq kd/2}
\binom{kd}{j} \binom{kd}{j-1} \exp\left(-\frac{j (kd-j+1) d}{n} \right) \right) \\
								& = & O\left( \sum_{1\leq j\leq
kd/2} \binom{kd}{j}^2 \exp\left(-\frac{j (kd-j) d}{n} \right) \right) \\ 
								& = & O\left( \sum_{1\leq j\leq
kd/2} \left(k^2 d^2 \exp\left(-\frac{(kd-j) d}{n} \right) \right)^j \right) \\ 
								& = & O\left( \sum_{j=1}^\infty
\left(k^2 d^2 \exp\left(-\frac{ k d^2}{2n} \right) \right)^j \right) \;.
\end{eqnarray*}
The lower bound on $c$ implies that $k^2 d^2 \exp\left(-\frac{ k d^2}{2n} \right)=O((kd)^{-1})<1$. So
%So the term $k^2 d^2
%\exp(-k d^2/2n)\leq (kd)^{-1}=o(1)$ as $d\tends +\infty$ with $n$ when, say, $c\geq 0.$ 
$$
\Pr(\mbox{There is no perfect matching in $H$})  = O\left(k^2 d^2 \exp\left(-\frac{ k d^2}{2n}
\right)\right) = O((kd)^{-1}) = o(1)\;. \qedhere 
$$
\end{proof}

The above argument does not work when we merely assume that $c\rightarrow +\infty.$ The forth and
final task in the proof of Theorem~\ref{G[A,ga(y)]} is to
adapt the argument provided by Erd\H{o}s and R\'enyi in~\cite{Erdos-Renyi1964} to the induced subgraph $H$. 

The key is to consider the minimum out and in degrees. Lemma~\ref{delta^-(H)}
and Lemma~\ref{delta^+(H)} combined imply that $\Pr(\{\delta^-(H)=1
\vee \delta^+(H)=0\}) =o(1)$ when $c\tends +\infty$. So
\begin{eqnarray*}
\Pr(\mbox{There is no matching in $H$}) &  =   & \Pr(\mbox{There is no matching in $H$} \wedge
\delta^-(H)>1 \wedge \delta^+(H)>0) + \\
								    & ~   & \Pr(\mbox{There is no
matching in $H$} \wedge \{\delta^-(H)=1 \vee \delta^+(H)=0\})\\
								   &\leq& \Pr(\mbox{There is no
matching in $H$}\wedge  \delta^-(H)>1 \wedge \delta^+(H)>0) + \\
								   &  ~  &\Pr(\delta^-(H)=1 \vee
\delta^+(H)=0)\\ 
								   &  =   & \Pr(\mbox{There is no
matching in $H$} \wedge  \delta^-(H)>1 \wedge \delta^+(H)>0) + o(1)\;.
\end{eqnarray*}
So we are only left with proving that $\Pr(\mbox{There is no matching in $H$} \wedge  \delta^-(H)>1
\wedge \delta^+(H)>0) = o(1)$ when $c\tends +\infty$ and $c\leq 5 \log(kd).$ In fact we prove
something slightly stronger.

\begin{proposition}\label{Min degree and matching}
Let $H$ be the graph introduced in Theorem~\ref{G[A,ga(y)]} and 
 $$
 c = \frac{kd^2}{n} - \log(kd)\;.
 $$
 Suppose that $0\leq c \leq 5\log(kd).$ Then
$$
\Pr(\mbox{There is no perfect matching in $H$} \wedge \delta^-(H)>1 \wedge \delta^+(H)>0) = o(1)\;.
$$
\end{proposition}
\begin{proof}
We once again apply Lemma~\ref{FK}:  no perfect matching exists in $H$ if there are non-empty sets
$S\subseteq A\setminus\{y\}$ and $T\subseteq \ga(y)$ such that $|S|+|T| = kd+1$ and $\ga(S)\cap T =
\emptyset.$ We consider the cases $|S|\leq |T|$ and $|T|<|S|$ separately. Let us start with the
former. 

Note that $\delta^+(H)>0$ implies that if $S$ belongs to a problematic pair, then $|S|>1$. Suppose for a contradiction that $|S|=1$. Then $T=\ga(y)$ and, since $(S,T)$ is problematic, there must be no edges between $S$ and $T$ and hence $\delta^+(H)=0$. %This observation will be crucial.

The size of $S$ lies in the range $2\leq |S| \leq (kd+1)/2$ since $|S|\leq |T|$. The key is to only consider $S$-minimal problematic pairs $(S,T)$. This means that there exists no
proper subset $S'\subsetneq S\subseteq A$ and $T '\subseteq \ga(y)$ with $|S'|+|T'|=kd+1$ such
that $(S',T')$ is a problematic pair. %For the case $|S|\leq |T|$ it is enough to consider $|S|$-minimal problematic pairs. since, if there exists no such pair, there exists no problematic pair. 

We crucially observe that every $w\in \ga(y)\setminus T$ must have at least two edges landing in
$S.$ Otherwise, picking a $w\in \ga(y)\setminus T$ and $s\in S$ such that $E(S,\{w\})=\{sw\}$ and
replacing $(S,T)$ by $(S\setminus\{s\}, T \cup \{w\})$ gives another problematic pair, which contradicts the
minimality of $S$.

Keeping all this in mind let us calculate the probability that such a problematic pair of sets
$(S,T)$ exists. We fix $S\subseteq A\setminus \{y\}$ and $T\subseteq \ga(y)$ and let $j=|S|$. We
have to bound from above the probability that there are no edges from $S$ to $T$ and that all the
vertices in $\ga(y)\setminus T$ have at least two edges starting in $S$. To keep the notation simple
let us write $\ga(y) \setminus T = \{w_1,\dots, w_{j-1}\}$ and also name some events: 
$$
E_0 =\{ \gai(T)\cap S= \emptyset\}
$$
and for $1\leq i \leq j-1,$
$$
 E_i = \bigwedge_{\ell=1}^i \{|\gai(w_{\ell}) \cap S| \geq 2\} \;.
$$
In this notation we have to bound 
\begin{eqnarray}\label{E_0 and E_j-1}
\Pr(E_0 \wedge E_{j-1}) = \Pr(E_0) \prod_{i=1}^{j-1} \Pr(E_i \mid E_0 \wedge E_{i-1})\;.
\end{eqnarray}
Recall that Proposition~\ref{No edge from S to T} states that for a fixed pair of sets $(S,T)$
where $|S|=j,$ the probability that $E_0$ occurs is bounded above by
\begin{eqnarray*}%\label{E_0}
\Pr(E_0) \leq (1+o(1)) \exp\left(-\frac{d j (kd+1-j)}{n}\right)\;. 
\end{eqnarray*}
So we are left to bound
\begin{eqnarray}\label{E_i given E_0 and E_i-1}
 \Pr(E_i \mid E_0 \wedge E_{i-1})&=& 	\Pr(|\gai(w_{i}) \cap S| \geq 2 \mid
	E_0 \wedge E_{i-1}) \nonumber\\
&\leq& \sum_{s\neq s' \in S} \Pr(\{s,s'\} \subseteq \gai(w_i) \mid
E_0 \wedge E_{i-1})\;.
\end{eqnarray}
Edges in $G$ are chosen with probability $d/n.$ If they were chosen independently, then the right
hand side would be $\tbinom{j}{2} (d/n)^2.$ We show that a similar upper bound holds for \eqref{E_i
given E_0 and E_i-1}.

\begin{lemma}\label{E_i given E_i-1}
Let $s$ and $s'$ be two distinct vertices in $S$ and $1\leq i \leq j-1$ be an integer. In the
notation established above we have  
$$
\Pr(\{s,s'\} \subseteq \gai(w_i) \mid E_0 \wedge E_{i-1}) \leq \left(\frac{d}{n-d}\right)^2\;.
$$ 
\end{lemma}
\begin{proof}
As
$$
\Pr(\{s,s'\} \subseteq \gai(w_i) \mid E_0 \wedge E_{i-1}) 
$$
equals
\begin{eqnarray}\label{product}
\Pr(s \in \gai(w_i) \mid E_0 \wedge E_{i-1}) \; \Pr(s' \in \gai(w_i) \mid E_0 \wedge E_{i-1} \wedge
s \in \gai(w_i))\;,
\end{eqnarray}
it is enough to prove that both probabilities appearing in the product above are at most $d/(n-d).$ Let us
temporarily set $W_i = \{w_1,\dots, w_{i-1}\}.$ 

To bound the first probability in \eqref{product} we observe that
\begin{eqnarray*}
kd &  =   & \E(d^+(s) \mid E_0 \wedge E_{i-1})   \\
     &=&  \sum_{z\in Z\setminus T} \Pr(z\in \ga(s) \mid E_0 \wedge E_{i-1}) \\
     &\geq&  \sum_{z\in Z\setminus (T\cup W_{i-1})} \Pr(z\in \ga(s) \mid E_0 \wedge E_{i-1}) \\
     &  =   & (kn-kd + j-i) \Pr(w_i \in \ga(s) \mid E_0 \wedge E_{i-1}) \\
    &\geq& (kn - kd) \Pr(s \in \gai(w_i) \mid E_0 \wedge E_{i-1}) \;.
\end{eqnarray*}
%  \begin{eqnarray*}
% kd &  =   & \E(d^+(s) \mid E_0 \wedge E_{i-1})   \\
%      &=&  \sum_{z\in Z\setminus T} \Pr(z\in \ga(s) \mid E_0 \wedge E_{i-1}) \\
%      &\geq&  \sum_{z\in Z\setminus \ga(y)} \Pr(z\in \ga(s) \mid E_0 \wedge E_{i-1}) \\
%      &=& (kn - kd) \Pr(s \in \gai(w_i) \mid E_0 \wedge E_{i-1}) \;.
% \end{eqnarray*}
The third (and final) equality following from the symmetry of random biregular biparite graphs. 

For the second probability in \eqref{product} we proceed similarly 
\begin{eqnarray*}
kd &  =   & \E(d^+(s') \mid E_0 \wedge E_{i-1}  \wedge s \in \gai(w_i))   \\
     &\geq& \sum_{z\in Z\setminus (T\cup W_{i-1})} \Pr(z\in \ga(s') \mid E_0 \wedge E_{i-1}  \wedge
s \in \gai(w_i)) \;.
\end{eqnarray*}
In this case we can not deduce the desired bound from the symmetry of random biregular bipartite graphs since not all $z\in Z\setminus (T\cup W_{i-1})$ have the same role in the graph.
Instead we prove via switching that the probability appearing in the sum above is minimal when $z=w_i.$ 
\begin{claim}
$$
 \Pr(w_i\in \ga(s') \mid E_0 \wedge E_{i-1}  \wedge s \in \gai(w_i))  \leq  \Pr(z\in \ga(s') \mid
E_0 \wedge E_{i-1}  \wedge s \in \gai(w_i)) \;,
$$
for all $z\in Z\setminus (T\cup W_{i-1}).$ 
\end{claim}
Let us quickly deduce the required inequality for the second probability appearing in \eqref{product} before proving the claim:
\begin{eqnarray*}
kd  &\geq& (kn-kd + j-i) \Pr(w_i \in \ga(s') \mid E_0 \wedge E_{i-1}  \wedge s \in \gai(w_i)) \\
    &\geq& (kn - kd) \Pr(s' \in \gai(w_i) \mid E_0 \wedge E_{i-1}  \wedge s \in \gai(w_i))\,.
\end{eqnarray*}

\begin{proof}[Proof of the claim]
Subtracting the probability 
$$
 \Pr(\{w_i,z\}\subseteq \ga(s') \mid E_0 \wedge E_{i-1}  \wedge s \in \gai(w_i)) 
$$
from both sides of the inequality we see that we have to prove that 
$$
 \Pr(w_i \in \ga(s') \wedge z\notin \ga(s') \mid E_0 \wedge E_{i-1}  \wedge s \in \gai(w_i)) 
$$
is at most
$$ 
\Pr(w_i \notin \ga(s') \wedge z\in \ga(s') \mid E_0 \wedge E_{i-1}  \wedge s \in \gai(w_i))\;.
$$
This is equivalent to proving that $|\G_{w_i}|\leq |\G_z|$ where 
$$
\G_{w_i} = \{G\in \G(k,n,d) : w_i\in \ga(s') \wedge z\notin \ga(s')\wedge E_0 \wedge E_{i-1} \wedge s \in \gai(w_i)\}
$$
and 
$$
\G_{z} = \{G\in \G(k,n,d) : z\in \ga(s') \wedge w_i\notin \ga(s')  \wedge E_0 \wedge E_{i-1} \wedge s\in \gai(w_i)\}\;.
$$
Like in the proof of Lemma~\ref{u vs in gai(T')}, we partition the two families of graphs in
subfamilies according to the
size of the intersection $\gai(w_i)\cap \gai(z)$. For any $0\leq\ell\leq d-1$ we define
$$
\G_{w_i,\ell } = \{G \in \G(k,n,d): w_i\in \ga(s') \wedge z\notin \ga(s')\wedge E_0 \wedge E_{i-1} \wedge s \in \gai(w_i) \wedge |\gai(w_i)\cap \gai(z)| = \ell\}
$$
and 
$$
\G_{z, \ell} = \{G \in \G(k,n,d): z\in \ga(s') \wedge w_i\notin \ga(s')  \wedge E_0 \wedge E_{i-1} \wedge s\in \gai(w_i) \wedge |\gai(w_i)\cap \gai(z)| = \ell\}\;.
$$
The parameter $\ell$ is at most $d-1$ in both $\G_{w_i}$ and $\G_z$ as $s'$ lies in exactly one of the two sets $\gai(w_i)$ and
$\gai(z)$. For $0\leq\ell\leq d-1$ we count in two ways $N_\ell,$ the number of switchings (introduced below the proof of Lemma~\ref{Expectation of intersection} in
p.~\pageref{Switching}) between $\G_{w_i,\ell}$ and $\G_{z,\ell}.$
In other
words we double count the number of ordered pairs $(G_{w_i},G_z) \in \G_{w_i,\ell} \times
\G_{z,\ell}$ such that $G_{w_i}$ is a switching of $G_z$ or equivalently that $G_z$ is a switching
of $G_{w_i}.$ 
\begin{figure}[ht]
 \begin{center}
 \includegraphics[width=\textwidth]{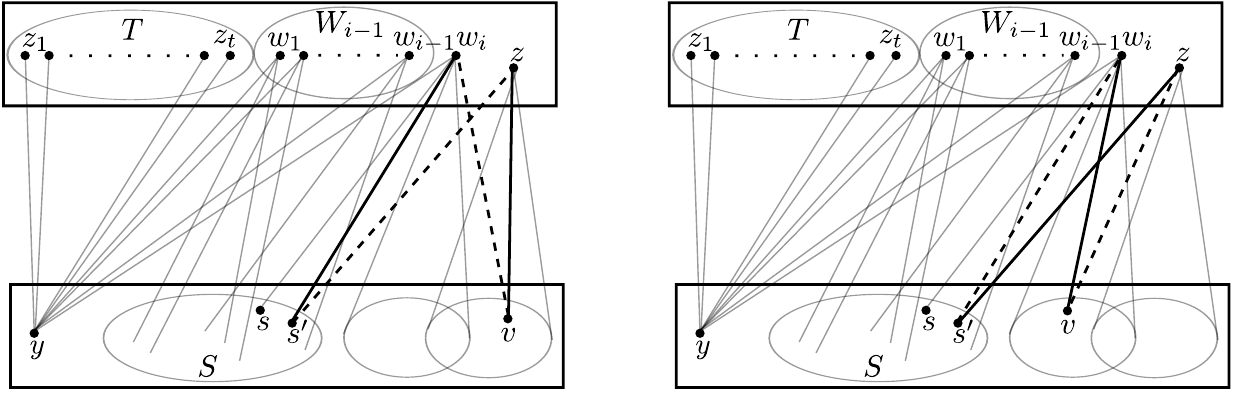}
 \end{center}
 \caption{A graph $G_{w_i}\in\G_{w_i}$ and its switching $G_z\in\G_z$. Solid lines represent edges
and dashed lines missing edges. }
 \label{fig:sw3}
\end{figure}

Take $G_{w_i}\in \G_{w_i,\ell}$ and $v\in \gai(z)\setminus \gai(w_i).$ Applying the $\{s'w_i, vz\}$-switching 
to $G_{w_i}$ results in a member of $ \G_{z,\ell}$ as the switching does not affect
any of the events $E_0, E_{i-1},   \{s\in \gai(w_i)\}$ and $\{|\gai(w_i)\cap \gai(z)| = \ell\}$ (see
Figure~\ref{fig:sw3}). There are
$(d-\ell)$ such $v$ and so 
$$
N_\ell = (d-\ell) |\G_{w_i,\ell}| \;.
$$
Now take $G_z\in \G_{z,\ell}$ and $ v\in \gai(w_i)\setminus (\gai(z) \cup \{y,s\}).$ Applying the
$\{s'z,vw_i\}$-switching to $G_z$ results in a member of $ \G_{w_i,\ell}$ as the switching does
not affect any of the events $E_0, E_{i-1},  \{s\in \gai(w_i)\}$ and $\{|\gai(w_i)\cap \gai(z)| =
\ell\}.$
There are at most $(d-\ell)$ such $v$ and so 
$$
N_\ell \leq (d-\ell) |\G_{z,\ell}| \;.
$$
Thus $|\G_{w_i,\ell}| \leq |\G_{z,\ell}| $ for all $0\leq \ell \leq d-1$ and
$$
 |\G_{w_i}| = \sum_{\ell=0}^{d-1}  |\G_{w_i,\ell}| \leq  \sum_{\ell=0}^{d-1}  |\G_{z,\ell}| = 
|\G_{z}|\;. \qedhere
$$
\end{proof}
This completes the proof of the lemma.
\end{proof}

We resume the proof of Proposition~\ref{Min degree and matching}. Inequality \eqref{E_i given E_0
and E_i-1} becomes
\begin{eqnarray*}
 \Pr(E_i \mid E_{i-1} \wedge E_0) \leq  \binom{j}{2} \left(\frac{d}{n-d}\right)^2\;.
\end{eqnarray*}
Substituting the above and the bound on $\Pr(E_0)$ from Proposition~\ref{No edge from S to T} in
\eqref{E_0 and E_j-1} gives
$$
\Pr(E_0 \wedge E_{j-1}) \leq (1+o(1)) \binom{j}{2}^{j-1} \left(\frac{d}{n-d}\right)^{2(j-1)}
\exp\left(-\frac{d j(kd+1-j)}{n}\right)\;.
$$
To bound the probability there is a $S$-minimal problematic pair with $|S|\leq |T|$ we apply a union
bound. For fixed $j$ there are at most $\tbinom{kd}{j} \, \tbinom{kd}{j-1}$ ways to choose
$(S,T)$ subject to $|S|=j.$ Thus the probability there is a $S$-minimal problematic pair with
$|S|\leq |T|$ is
\begin{eqnarray*}
O\left( \sum_{j=2}^{(kd+1)/2} \binom{kd}{j} \binom{kd}{j-1} \binom{j}{2}^{j-1} 
\left(\frac{d}{n-d}\right)^{2(j-1)} \exp\left(-\frac{d j(kd+1-j)}{n}\right) \right) \;.
\end{eqnarray*}
Applying the well known bound $\binom{a}{b} \leq (\tfrac{e a}{b})^b,$ using that $n-d\geq n/2$ as
$c=O(\log(kd))$ and writing $\ell=j-1$ we get that the probability there is a $S$-minimal
problematic pair with $|S|\leq |T|$ is  
% \begin{eqnarray*}
% O\left( \sum_{1\leq  \ell \leq kd/2} \left(C\, \frac{ k d^3 \,e^{ \ell d /n}}{n^2}\right)^{\ell} 
% \left(k d  e^{-kd^2/n}\right)^{\ell+ 1} \right)
% \end{eqnarray*}  
\begin{eqnarray*}
O\left( \sum_{\ell=1}^{kd/2} k d  e^{-kd^2/n} \left(C\, \frac{ k^2 d^4 \,e^{ -d
(kd-\ell) / n}}{n^2}\right)^{\ell} 
 \right)
\end{eqnarray*}
for some large enough constant $C$.

The definition of $c$ gives that $e^{-kd^2/n} = e^{-c} / kd.$ Observing that $kd-\ell\geq
kd/2$, the above sum is
\begin{eqnarray*}
& & O\left( e^{-c} \sum_{\ell=1}^{kd/2} \left(C\, \left(\frac{k d^2}{n}\right)^{2}
\,\frac{e^{-c/2}}{\sqrt{kd}} \right)^{\ell}  \right)\;.
\end{eqnarray*}
As $kd^2/n \leq 6 \log(kd)$ and $e^{-c}\leq 1$ we get that the probability a
problematic pair exists with $|S|\leq |T|$ is
\begin{eqnarray*}
& & O\left(\sum_{l =1}^\infty \left(C'\, \frac{\log^{2}{(kd)}}{\sqrt{kd}}
\right)^{\ell}  \right) 
= O\left(  \frac{\log^{2}{(kd)}}{\sqrt{kd}} \right) 
=  o(1)\;,
\end{eqnarray*}  
where $C'$ is another large enough constant. 

The case when $|T|<|S|$ is similar. This time we chose $T$-minimal problematic pairs. The set $T$
cannot be a singleton as $\delta^-(H)>1.$ The minimality of $|T|$ implies that for every $v\in
A\setminus S$ there are at least two edges starting at $v$ and ending in $T.$ The calculations
needed are very similar to those given above and are omitted. 

\end{proof}
 
Let us quickly recap the proof of Theorem~\ref{G[A,ga(y)]}.
\begin{proof}[Proof of Theorem~\ref{G[A,ga(y)]}]
When $c\tends -\infty$ Lemma~\ref{delta^-(H)} gives that there is no matching in $H$
\textit{whp}.

When $c\tends +\infty$ Lemma~\ref{delta^-(H)} and Lemma~\ref{delta^+(H)} give that
$\Pr(\delta^-(H)>1 \vee \delta^+(H)>0) = 1-o(1).$ Proposition~\ref{Min degree and matching} states
that the probability there is no matching and $\delta^+(H)>0$ or $\delta^-(H)>1$ is $o(1)$ provided
that $c\leq 5 \log(kd).$ So there is a matching in $H$ \whp when $c\tends +\infty$ and $c\leq 5
\log(kd).$ Finally Proposition~\ref{Large c} gives that there is a matching in $H$ \whp when $c\geq 5
\log(kd).$ 
\end{proof} 
 
We conclude the section with some remarks on the probability of the events $A^+_{y'}$ defined in
\eqref{A^+_y'} at the beginning of the proof of Lemma~\ref{delta^+(H)} on p.\pageref{A^+_y'}.
Suppose for a moment that $G$ was chosen uniformly at random from the family of $d^+$-regular bipartite
graphs. In other words $d^+$ would be constant (and equal to $kd$), but no restriction on $d^-$
would exist.
Then the neighbourhoods of vertices in $Y$ would be chosen uniformly at random from all $(kd)$-element subsets
of $Z$ and independently of each other. So the probability that $\ga(y)\cap\ga(y')=\emptyset$ would
be equal to $\tbinom{kn-kd}{kd} /\tbinom{kn}{kd}.$ 

The probability that $A^+_{y'}$ occurs increases in random biregular bipartite graphs as the event $\{\ga(y) = S_1 \wedge
\ga(y')= S_2\}$ is more likely when $S_1\cap S_2 = \emptyset$ than when $S_1\cap S_2 \neq
\emptyset.$ This can be proved via switching and is due to the fact that vertices in $\ga(y)$ have
one of their $d$ incoming edges ``taken up'' by $y.$ Thus the edges coming out of
$y'$ are more likely to land in $Z\setminus\ga(y).$ We do not give the details of the proof here as
a lower bound on $\Pr(A^+_{y'})$ is not necessary. We only state the lower bound and compare it with the
upper bound coming from Proposition~\ref{No edge from S to T}:
\begin{eqnarray}\label{upper and lower bound}
\frac{\binom{kn-kd}{kd}}{\binom{kn}{kd}} \leq \Pr(A^+_{y'}) &\leq& \left(
\frac{\binom{n-2}{d-1}}{\binom{n-1}{d-1}}\right)^{kd}\;.
\end{eqnarray}
For $d=o(n^{2/3})$ both bounds are asymptotically equal to $\exp(-kd^2/n)$, which is easy to see
using Stirling approximation to the binomial coefficients. This reinforces the idea that the
dependence among small sets of edges in $G(k,n,d)$ is small.

%  
%  \comment{Shall we change the name of the set $A$ or of the random variables $A^\pm$?}
%  \guillem{Perhaps the variables?, but now I don't have any non ambiguous name in mind. $C$ or
% $D$?}
%  \giorgis{It might not be necessary. When reading the proof this time round it did not seem like
% contradicting notation. I was thinking to try $X,$ but this poses somewhat similar problems! I
% guess % we can leave it as it is}
%  

 \section{Remarks on Theorem~\ref{G[A,B]}}\label{Thm 2}

We only outline the proof of Theorem~\ref{G[A,B]} as it is very similar to that of Theorem~\ref{G[A,ga(y)]}. The difference lies in the induced subgraph under consideration. For the former $H$ is
defined to be $G[A,B]$ where $A\subseteq Y$ and $B\subseteq Z$ are sets of size $kd.$ For the latter $B$ is taken to be the neighbourhood of some $y\in A.$ This
complicates some parts of the proof and is why we opted to give the proof of
Theorem~\ref{G[A,ga(y)]}.

When $d=o(\sqrt{n})$ it is straightforward to show there is no matching in $H$ \textit{whp}. By
Lemma~\ref{Expectation of intersection} we know that for any $y\in A$ the expected value
$\E(|\ga(y)\cap B|) =o (1).$ Thus the probability $\Pr(\ga(y) \cap B = \emptyset) = 1-o(1)$ and
consequently there is no matching in $H$ \textit{whp}. 

The first step in dealing with larger values of $d$ is to prove a variation of Proposition~\ref{No
edge from S to T}. As $y$ no longer has a special role it is possible to bound the probability there
are no edges from $S$ to $T$ by looking one by one at the vertices of $S$ or $T.$ In
Proposition~\ref{No edge from S to T} we only worked with the vertices in $T$.  
 
\begin{proposition} \label{Second no edge from S to T}
Let $Y,A, Z,B $ and $G$ be like in the statement of Theorem~\ref{G[A,B]}. Let $S\subseteq A$ and
$T\subseteq  B.$  

Suppose that $z\in T$ and $|S|+d\leq n$. Then 
\begin{eqnarray*}
\Pr(\Gamma(S)\cap T = \emptyset) &\leq&\Pr(\gai(z)\cap S = \emptyset )^{|T|}\\
&\leq& \left(1-\frac{|S|}{n}\right)^{|T|}
\left(1-\frac{|S|}{n-1}\right)^{|T|} \dots \left(1-\frac{|S|}{n-d+1}\right)^{|T|} \\
&\leq& (1+o(1)) \exp{\left(-\frac{d\,|S|\,|T|}{n}\right)}\;.
\end{eqnarray*}
Suppose that $y\in S$ and $|T|+kd\leq kn$. Then
\begin{eqnarray*}
\Pr(\Gamma(S)\cap T = \emptyset) &\leq&\Pr(\ga(y)\cap T = \emptyset )^{|S|}\\
&\leq& \left(1-\frac{|T|}{kn}\right)^{|S|}
\left(1-\frac{|T|}{kn-1}\right)^{|S|} \dots \left(1-\frac{|T|}{kn-kd+1}\right)^{|S|} \\		
&\leq& (1+o(1)) \exp{\left(-\frac{d\,|S|\,|T|}{n}\right)}\;.
\end{eqnarray*}
\end{proposition}
\begin{proof}[Sketch of proof]
For $z\in Z$ the probability there are no edges from $S$ to $z$ equals
$$
\Pr(\gai(z)\cap S = \emptyset ) = \frac{\binom{n-|S|}{d}}{\binom{n}{d}} =
\left(1-\frac{|S|}{n}\right) \left(1-\frac{|S|}{n-1}\right) \dots
\left(1-\frac{|S|}{n-d+1}\right)
$$  
as $\gai(z)$ is chosen uniformly at random from all $d$-element subsets of $Y\setminus S.$

Now let $T=\{z_1,\dots,z_t\}.$ It can be shown via a switching argument very similar to that in the
proof of Proposition~\ref{No edge from S to T} that for $2\leq i\leq t$
$$
\Pr(\gai(z_i) \cap S =\emptyset \mid \gai(\{z_1,\dots,z_{i-1}\}) \cap S  = \emptyset ) \leq
\Pr(\gai(z_i) \cap S =\emptyset) = \Pr(\gai(z) \cap S =\emptyset)\;.
$$
This leads to
$$
\Pr(\ga(S) \cap T = \emptyset) \leq \Pr(\gai(z)\cap S = \emptyset )^{|T|} \;.
$$
A similar approach is applied for the second claim. 
\end{proof}

Next we prove a variation of Lemma~\ref{delta^-(H)} for the minimum degree of $H$, $\delta(H) =
\min\{\delta^+(H),\delta^-(H)\}.$ We no longer need to distinguish between $\delta^+(H)$ and
$\delta^-(H)$ since $B\subseteq Z$ is an arbitrary set.

 \begin{lemma}\label{delta^(H)}
 Let $H$ be the graph introduced in Theorem~\ref{G[A,B]} and 
 $$
 c = \frac{kd^2}{n} - \log(kd)\;.
 $$
 Then
 \begin{enumerate}
\item[(i)] $\delta(H)=0$ \whp when  $c \rightarrow - \infty$ or when $d$ is a constant.
\item[(ii)] $\delta(H)>0$ \whp when $c \rightarrow +\infty.$ 
\end{enumerate}
In particular there is no perfect matching in $H$ \whp when $c\rightarrow -\infty.$
 \end{lemma}
\begin{proof}[Sketch of proof]
We consider two types of events: 
\begin{eqnarray*} 
B^+_y = \{\ga(y)\cap B = \emptyset\}\qquad \mbox{ for $y\in A$}
\end{eqnarray*}
and 
\begin{eqnarray*} 
B^-_z = \{\gai(z)\cap A = \emptyset\}\qquad \mbox{ for $z\in B$}\;.
\end{eqnarray*}
We also define the random variables
\begin{eqnarray*}
Q^+ &=& \sum_{y\in A} 1_{B^+_y}\;,\\
Q^- &=& \sum_{z\in B} 1_{B^-_z}\;\text{ and}\\
Q   &=& Q^+ + Q^-\;.
\end{eqnarray*}
The condition $\delta(H)>0$ holds if and only if $Q=0.$

The probability that $B^+_y$ occurs equals
$$
\Pr(B^+_y)= \frac{\binom{kn-kd}{kd}}{\binom{kn}{kd}}\qquad \mbox{ for all $y\in A$}\;,
$$   
as the neighbourhood of $y$ is chosen uniformly from all $(kd)$-elements subsets of $Z.$ Similarly
$$
\Pr(B^-_z)= \frac{\binom{n-kd}{d}}{\binom{n}{d}}\qquad \mbox{ for all $z\in B$}\;.
$$   
When $d=o(n^{2/3})$
$$
\Pr(B^+_y), \Pr(B^-_z)= \Theta\left( \exp\left(-\frac{kd^2}{n}\right)\right)\;.
$$
% \guillem{We have to take into account what happen if $d=o(n^{2/3})$ or not.}
% \giorgis{Unfortunately I think we do. I will double check this}

So
$$
\E(Q) = kd (\Pr(B^+_y)+\Pr(B^-_z)) =  \Theta\left( kd \exp\left(-\frac{kd^2}{n}\right)\right) =
\Theta(e^{-c})\;.
$$
In particular $\E(Q)=o(1)$ if $c\tends +\infty$ and $d=o(n^{2/3})$. If $c\tends +\infty$, but $d$ is
not $o(n^{2/3})$ it is easy to check that $\E(Q) =o(1)$. The
second conclusion follows.

If $c\tends -\infty$, it is adequate to prove that $\Pr(Q^-=0)=o(1)=\Pr(Q^+=0)$. For this we apply
Lemma~\ref{Chebyshev} (Chebyshev's inequality). The upper bound $\var(Q^-)\leq \E(Q^-)$ and $\var(Q^+)\leq \E(Q^+)$ derived in the proof of
Lemma~\ref{delta^-(H)} holds as Proposition~\ref{Second no edge from S to T} gives that $\Pr(B^-_z\wedge B^-_{z'}) \leq
\Pr(B^-_z)^2$ for $z,z'\in Z$ and $\Pr(B^+_y\wedge B^+_{y'}) \leq \Pr(B^+_y)^2$ for $y,y'\in Y.$ 
\end{proof} 

Having proved the first claim of Theorem~\ref{G[A,B]} we proceed to the second. For $c\geq 5
\log(kd)$ we apply Proposition~\ref{Second no edge from S to T} in the way described in the
proof of Proposition~\ref{Large c} to get that there is a matching in $H$ \textit{whp}.

We are only left with showing that when $c\rightarrow +\infty$ and $c\leq 5 \log(kd)$  the
probability
$$
\Pr(\mbox{There is no perfect matching in $H$} \wedge \delta(H)>0)=o(1)\;.
$$
This can be done in a very similar way to the proof of Proposition~\ref{Min degree and matching}.
Some amendments have to be made, for example one has to consider pairs of sets $(S,T)$ where $S\subseteq
A$ and not $A\setminus \{y\}$.

% \comment{Check that section to make sure of this}
% \guillem{I think it can be extended in the natural way.}

We conclude the section with a quick explanation as to why our method as presented is not strong enough to yield an
(asymptotically) exact expression for the probability that there is a matching in $H;$ something
that Erd\H{o}s and R\'enyi achieved for $B(n,p).$  

As we have seen it is enough to get an asymptotically exact value for the probability
$\Pr(\delta(H)=0).$
This is equivalent to none of the events $B^+_y$ or $B^-_z$ occurring. Erd\H{o}s and R\'enyi used
the inclusion-exclusion principle and exact expressions for the probability of events like
$$
\bigwedge_{y\in S} B^+_y \wedge \bigwedge_{z\in T} B^-_z\;,
$$   
where $S\subseteq A$ and $T\subseteq B$.

It is hard to obtain exact expressions for the probability of this kind of events because of the
lack of independence in choosing the edges in $H.$ The switching double counting method can be
applied to give upper bounds, which appear to be reasonably sharp. Obtaining lower bounds,
like the one in~\eqref{upper and lower bound}, seems to be harder.

% We conclude the section with a quick explanation as to why there always exists a perfect matching
% in
% $G\sim G(1,n,d)$. By Lemma~\ref{FK} it is sufficient to show that $|\ga(S)|\geq |S|$ for all
% non-empty subsets $S\subseteq Y.$ This is easily shown by counting in two ways the number of edges
% from $S$ to $\ga(S):$
% $$
% d |S| = |E(S,\ga(S))| \leq d |\ga(S)|\;.
% $$  
% 
% 
% \guillem{I don't know if this last paragraph is needed in the sense that it is a well-known
% result.
% But if you want to keep it, I have no problem.} \giorgis{Let's find a reference and remove it!}
% 

\section{Commutative graphs}\label{Commutative}

In this final section we apply the results obtained in Section~\ref{Thm 3} to prove
Theorem~\ref{theorem Commutative}. 

% To cater for the direction of the edges we introduce the notion of the \emph{inverse} of a
% directed layered graph $G$ with $V(G)=V_0\cup\dots\cup V_h$ \cite{Nathanson1996, GPPlIn}. The
% layers of $I$  are $V_h\cup\dots\cup V_0$  and its edges are determined by $vw\in E(I)$ if and
% only if $wv\in  E(G).$ Observe that, $d^+_I=d^-_G$ and $d^-_I=d^+_G.$  
% 
% For Theorem~\ref{Lower commutative} we apply the first part of Theorem~\ref{G[A,ga(y)]}.

\begin{proof}[Proof of Theorem~\ref{theorem Commutative}~$(i)$]
We show that the upper bound on $d$ implies that Pl\"unnecke's upward condition is
violated \whp for all edges in $E(X_{h-2},X_{h-1}).$ So, for $xy\in E(X_{h-2},X_{h-1})$, we show that
\whp there is no perfect matching in $G$ from $\ga(x)$ to $\ga(y).$

We apply the first part of Theorem~\ref{G[A,ga(y)]} with
$A=\ga(x)\subseteq X_{h-1} =Y$ and $\ga(y)\subseteq X_{h}=Z$. Hence $n=|X_{h-1}|=k^{h-1}m$,
$d^+=kd$ and $d^-=d$. The parameters satisfy the condition $kd\leq n$ as $d\leq m$. By Theorem~\ref{G[A,ga(y)]} there is no perfect matching in $G$ from $\ga(x)$ to
$\ga(y)$
\whp provided that
$$
\frac{kd^2}{n}- \log{kd} = \frac{d^2}{k^{h-2}m}- \log{kd} \tends -\infty \;.
$$

Since it is easy to see that Theorem~\ref{theorem Commutative}~$(i)$ holds for $d=o(\sqrt{m})$ (see
Section~\ref{Thm 2}), we assume that $d\geq m^{1/2-\eps}$ for some small $\eps>0$, so that
 $\log{kd}\geq \left(\frac{1}{2}-\eps\right)\log{km}$. The non existence \whp of a perfect
matching between $\ga(x)$ and $\ga(y)$ is implied by the condition
$$
\frac{d^2}{k^{h-2}m}- \left(\frac{1}{2}-\eps\right)\log{km} \tends -\infty \;.
% d\leq \sqrt{\left(\frac{1}{2}-\eps\right)k^{h-2}m(\log{k^{h-1}m}-\omega(1))}
$$
This is in turn implied by the condition
$$
d\leq \sqrt{\frac{1}{3}k^{h-2}m\log{k m}}
$$
and the proof is concluded.

% \guillem{At some step, we require that $k\geq 1$ which is a natural condition, I think.
% Perhaps we should state it in the theorems}
% We show the condition on $d$ implies the stronger statement that Pl\"unnecke's downward condition
% is violated \whp for all edges $yz\in E(Y,Z).$ 
% 
% Let $yz$ be such an edge. We will work in $I,$ the inverse graph of $G$. Pl\"unnecke's downward
% condition for the edge $yz$ is satisfied in $G$ if and only if there is a matching in $I$ from
% $\ga(z)$ to $\ga(y).$ 

% Applying Theorem~\ref{G[A,ga(y)]} for $A=\ga(z)$ (having noted that $y\in
% A, n=km, d^+=d$ and $d^-=kd$) gives that there is no matching in $I$ from $\ga_I(z)$ to $\ga_I(y)$
% \whp provided that $d^2/m -\log{kd} \tends -\infty.$
\end{proof}

For Theorem~\ref{theorem Commutative}~$(ii)$ we rely on Proposition~\ref{Large c}.

\begin{proof}[Proof of Theorem~\ref{theorem Commutative}~$(ii)$]
For $1\leq j \leq h-1$ set $G_j'=G[X_{j-1},X_{j},X_{j+1}]$ to be the induced subgraph of $G$ on the
vertex set $X_{j-1}\cup X_j\cup X_{j+1}.$ We will calculate the probability that Pl\"unnecke's
conditions are not satisfied in $G_j'$ and then apply a union bound. 

Let $xy\in E(X_{j-1},X_j)$ and $H=G[\ga(x),\ga(y)]$ be the induced subgraph on the vertex set
$(\ga(x),\ga(y))$. Then 
\begin{eqnarray*}
\Pr(\mbox{Pl\"unnecke's upward condition is violated for $xy$}) & = & \Pr(\mbox{There is no perfect matching in $H$}) \\
												  &= & O\left( k^2d^2  \exp \left(- \frac{d^2}{2 k^{j-1} m}\right)\right)\;,
\end{eqnarray*}											
as we see by applying Proposition~\ref{Large c} and noting that the condition on $c$ is satisfied.
There are $k^{j} dm$ such edges and so  
\begin{eqnarray*}
\Pr(\mbox{Pl\"unnecke's upward condition is violated in $G_j'$})  & = & O\left( k^{j+2} d^3 m \exp
\left(- \frac{d^2}{2 k^{j-1} m}\right)\right)\;.
\end{eqnarray*}	

Let $yz\in E(X_{j},X_{j+1})$ and $I_j$ the graph obtained by
reversing the direction of the edges of $G_j'$ (called the \emph{inverse} of $G_j'$). It is easy to
see that Pl\"unnecke's downward condition for $yz$ in $G_j'$ is equivalent to Pl\"unnecke's
upward condition for $zy$ in $I_j$.
A similar calculation gives 
\begin{eqnarray*}
\Pr(\mbox{Pl\"unnecke's downward condition is violated for $yz$}) & = & O\left(d^2  \exp \left(-
\frac{d^2}{2 k^{j-1} m}\right)\right)\;.
\end{eqnarray*}
There are $k^{j+1} dm $ such edges and so   
\begin{eqnarray*}
\Pr(\mbox{Pl\"unnecke's downward condition is violated in $G_j'$})  & = & O\left( k^{j+1} d^3 m \exp
\left(- \frac{d^2}{2 k^{j-1} m}\right)\right)\;.
\end{eqnarray*}	
Adding the two probabilities gives
\begin{eqnarray*}
\Pr(\mbox{Pl\"unnecke's conditions are violated in $G_j'$})  & = & O\left( k^{j+2} d^3 m \exp
\left(-
\frac{d^2}{2 k^{j-1} m}\right)\right)\;.
\end{eqnarray*}	

The right hand side is an increasing function of $j$ as $k\geq 1$ and so  
\begin{eqnarray*}
\Pr(\mbox{Pl\"unnecke's conditions are violated in $G$})  &\leq& \sum_{j=1}^{h-1}
\Pr(\mbox{Pl\"unnecke's conditions are violated in $G_j'$})\\
											& = &
O\left( h k^{h+1} d^3 m \exp \left(-\frac{d^2}{2 k^{h-2} m}\right)\right)\\
											& =
&O\left(h k^{h+1}  m^4 \exp \left(- \frac{d^2}{2 k^{h-2} m}\right)\right)\\
											& = & O(m^{-1/2})\;,
\end{eqnarray*}	
when $d\geq 3 \sqrt{k^{h-2}m \log(h k^{h+1} m)}.$

\end{proof}   
The bounds on $d$ appearing in Theorem~\ref{theorem Commutative} have the same asymptotic order. They 
can be improved slightly, but as we were not able to obtain matching lower and upper bounds we opted to
present a proof as simple as possible. Note also that when $0 < k \leq 1$ one can obtain estimates
on the probability that $G$ is a commutative graph by applying Theorem~\ref{theorem Commutative} to
the inverse of $G.$

We conclude with some remarks linking the present results with those of \cite{GPPlIn}. For fixed $m$
and $k$ the lower bound on $d$ provided in Theorem~\ref{theorem Commutative} surpasses $m$ for
sufficiently large $h.$ This is of course not possible and implies that for a given $m$ and $d$
there is a limit to how large $h$ can be taken to be. This reflects the fact that infinite biregular
commutative graphs do not exist when $k>1$. It should also be noted that explicit constructions are
more economical in $m$ than probabilistic: a path is an infinite commutative graph with augmentation
1 and for integer $k>1$ there exists a commutative biregular graph with 3 layers and augmentation
$k$ whose bottom layer is a doubleton.

\bibliography{pluneke}
\bibliographystyle{amsplain}

 \end{document}